\documentclass{article} 

\usepackage{amsmath, amssymb, amsthm, amsfonts, bm, mathtools} 
\usepackage[normalem]{ulem}
\usepackage{graphicx, caption, xcolor, placeins} 
\usepackage[labelformat=simple]{subcaption}

\usepackage{tikz,pgfplots}
\usetikzlibrary{arrows,scopes} 
\pgfplotsset{compat=1.13}
\usetikzlibrary{arrows.meta,scopes} 
\tikzset{>={Straight Barb[round,angle=60:.12cm 1]}} 

\usepackage{authblk}

\newcommand{\sep}{\unskip,\ }
\newcommand{\MSC}[1][2010]{\noindent\textit{{#1} MSC:}}
\newenvironment{keyword}{\noindent\textit{Keywords:}}{}

\expandafter\def\csname ve\endcsname{\varepsilon} 

\title{Galerkin finite element methods for the numerical solution of two 
classical-Boussinesq type systems over variable bottom topography}
\author[1,2]{G. Kounadis}
\author[1,2]{D.~C. Antonopoulos}
\author[1,2]{V.~A. Dougalis \thanks{Corresponding author, email: 
{\tt doug@math.uoa.gr}}} 
\affil[1]{Department of Mathematics, National and Kapodistrian University of 
Athens, 15784 Zografou, Greece} 
\affil[2]{Institute of Applied and Computational Mathematics, FORTH, 70013 
Heraklion, Greece} 
\date{\vspace{-5ex}}

\newtheorem{theorem}{Theorem} 

\numberwithin{equation}{section} 
\numberwithin{theorem}{section} 

\newcommand{\hzr}{\smash{\stackrel{\raisebox{-5pt}{$\scriptscriptstyle 0$}}{H}}}
\DeclareMathOperator{\opP}{P} 
\DeclareMathOperator{\opR}{R} 

\begin{document} 
\maketitle 

\begin{abstract} 
We consider two `Classical' Boussinesq type systems modelling two-way 
propagation of long surface waves in a finite channel with variable bottom 
topography. Both systems are derived from the 1-d Serre-Green-Naghdi (SGN) 
system; one of them is valid for stronger bottom variations, and coincides with 
Peregrine's system, and the other is valid for smaller bottom variations. We 
discretize in the spatial variable simple initial-boundary-value problems 
(ibvp's) for both systems using standard Galerkin-finite element methods and 
prove $L^2$ error estimates for the ensuing semidiscrete approximations. We 
couple the schemes with the 4th order-accurate, explicit, classical Runge-Kutta 
time-stepping procedure and use the resulting fully discrete methods in 
numerical simulations of dispersive wave propagation over variable bottoms with 
several kinds of boundary conditions, including absorbing ones. We describe in 
detail the changes that solitary waves undergo when evolving under each system 
over a variety of variable-bottom environments. We assess the efficacy of both 
systems in approximating these flows by comparing the results of their 
simulations with each other, with simulations of the SGN-system, and with 
available experimental data from the literature. 
\end{abstract} 


\section{Introduction} \label{sec:1} 
The `Classical' Boussinesq system, \cite{W}, in one spatial dimension is the 
nonlinear, dispersive system of pde's 
\begin{equation}\label{eq:CB}\tag{CB} 
\begin{aligned} 
& \zeta_t + u_x + \varepsilon(\zeta u)_x = 0, \\ 
& u_t + \zeta_x + \varepsilon uu_x - \frac \mu 3 u_{xxt} = 0. 
\end{aligned} 
\end{equation} 
It has been derived, cf.\ e.g.\ \cite{W}, as an approximation of the 
two-dimensional Euler equations of water-wave theory, and models two-way 
propagation of long waves of small amplitude on the surface of an ideal fluid 
(say, water) in a horizontal channel of finite depth. The variables in 
\eqref{eq:CB} are nondimensional and scaled; $x$ and $t$ are proportional to 
length along the channel and time, respectively and the function $\varepsilon 
\zeta(x,t)$ represents the free surface elevation of the water above a level of 
rest at $z=0$. (Here $z$ is proportional to the depth variable and is taken 
positive upwards). The function $u=u(x,t)$ is the depth-averaged horizontal 
velocity of the fluid. The scaling parameters $\varepsilon$, $\mu$ are defined 
as $\varepsilon = \frac{A}{h_0}$, where $A$ is a typical amplitude of the 
surface wave and $h_0$ is the depth of the channel, and as 
$\mu=\frac{h_0^2}{\lambda^2}$, where $\lambda$ is a typical wavelength of the 
waves. The assumptions behind the derivation of \eqref{eq:CB} are that 
$\varepsilon\ll1$, $\mu\ll 1$, and that $\varepsilon$ and $\mu$ are related so 
that $\varepsilon = \mathcal O(\mu)$, i.e.\ are in the so-called Boussinesq 
scaling regime. The first pde in \eqref{eq:CB} is exact while the second is an 
$\mathcal O(\varepsilon^2)$ approximation to a relation obtained from the Euler 
equations. It is to be noted that in the variables of \eqref{eq:CB}, the 
horizontal bottom lies at $z=-1$. 

The initial-value problem for \eqref{eq:CB} with initial data $\zeta(x,0) = 
\zeta_0(x)$, $u(x,0)=u_0(x)$ on the real line has been studied by Schonbek 
\cite{S} and Amick \cite{Am}, who established global existence and uniqueness of 
smooth solutions under the assumption that $1+\varepsilon \inf_x \zeta_0(x) > 
0$. One conclusion of this theory is that for all $t\geq 0$, $1+\varepsilon 
\inf_x \zeta(x,t) > 0$, i.e.\ that there is always water in the channel. 
Existence-uniqueness of solutions globally in time in Sobolev spaces were 
established in \cite{BCS2}. The initial-boundary-value problem (ibvp) for 
\eqref{eq:CB} posed on a finite interval, say $[0,1]$, with zero boundary 
conditions for $u$ at $x=0$ and $x=1$, and no boundary conditions for $\zeta$, 
was proved in \cite{Ada11} to possess global weak (distributional) solutions. 

The system \eqref{eq:CB} has been used and solved numerically extensively in the 
engineering literature. We will refer here just to \cite{AD2} and \cite{AD1} 
for error estimates of Galerkin-finite element methods for the ibvp for 
\eqref{eq:CB} mentioned above and a computational study of the properties of the 
solitary-wave solutions of the system. For the numerical analysis of the 
periodic ivp we refer to \cite{ADMper}. 

In this paper we will be interested in the numerical solution of extensions of 
\eqref{eq:CB} valid in channels of variable-bottom topography. Several such 
extensions have been derived in the literature. Here we will follow \cite{LB} 
and consider two specific such variable-bottom models that may be derived from 
the Serre-Green-Naghdi \eqref{eq:SGN} system of equations, \cite{51, 52, GN}; 
for their derivation and theory of their validity we refer to \cite{LB} and 
\cite{La} and their references. 

\begin{figure}[htbp]
\centering 
\begin{tikzpicture} 
\draw[->] (1,-3) -- (1,1) node[left] {$z$}; 
\draw[->] (-.5,0) -- (8,0) node[right] {$x$}; 
\node[below left] at (1,0) {$0$}; 

\draw (-.5,-.25) .. controls (.25,-.5*3/4) and (.25,-.5*3/4) .. 
(1,0) .. controls (2,.5) and (2,.5) .. 
(3,0) .. controls (4,-.5) and (4,-.5) .. 
(5,0) .. controls (6,.5) and (6,.5) ..  
(7,0) .. controls (7.5,-.25) and (7.5,-.25) .. 
(8,-.25); 
\draw[<-] (2.7,.15) -- (3.2,.5) node[right] {$z=\varepsilon\zeta(x,t)$}; 

{[yshift=-2.5cm] 
\draw[dashed] (-.5,0) -- (8,0); 
\node[below left] at (1,0) {$-1$}; 
\draw (-.5,.25)  .. controls (-.5+1.25,.5) and (-.5+1.25,.5) ..
(2,0) .. controls (2+1.25,-.5) and (2+1.25,-.5) ..  
(4.5,0) .. controls (4.5+1.25,.5) and (4.5+1.25,.5) ..  
(7,0) .. controls (7+.5,-.5*.5/1.25) and (7+.5,-.5*.5/1.25) .. 
(8,-.25); 
\draw[<-] (6.25,.25) -- (5.25,-.35); 
\node[right] at (4.5,-.6) {$z = -\eta_b(x)\equiv -1+\beta b(x)$}; 
}

{[xshift=4cm] 
\draw[<->] (0,-2.7) -- (0,-.37); 
\node[right] at (0,-1.25) {$\eta(x,t) = \varepsilon\zeta(x,t) + \eta_b(x)$}; 
}
\end{tikzpicture} 
\caption{Scaled variables and variable-bottom topography. \label{fig:1}} 
\end{figure}
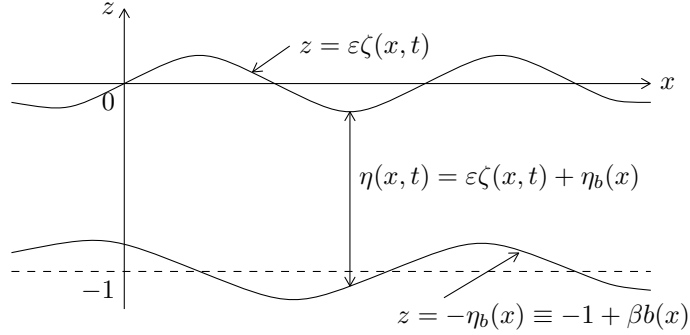 

In order to describe the topography of the bottom, in addition to $\varepsilon$ 
and $\mu$ we consider the scaling parameter $\beta$ defined by $\beta = 
\frac{B}{h_0}$, where $B$ is a typical bottom topography variation and $h_0$ is 
now a reference depth. In scaled nondimensional variables consistent with those 
in \eqref{eq:CB} the Serre-Green-Naghdi equations are written as 
\begin{equation}\label{eq:SGN}\tag{SGN} 
\begin{aligned} 
& \zeta_t + (\eta u)_x = 0, \\ 
& \begin{multlined}[.75\textwidth]  
\left(1+\frac \mu \eta \mathcal T[\eta, \beta b]\right)u_t + \zeta_x + 
\varepsilon uu_x \\ 
+ \mu \varepsilon\left\{-\frac{1}{3\eta}(\eta^3(uu_{xx} - (u_x)^2)_x + \mathcal 
Q[\eta,\beta b]u\right\} = 0, 
\end{multlined} 
\end{aligned} 
\end{equation} 
where the operators $\mathcal T[\eta,\beta b]$, $\mathcal Q[\eta,\beta b]$ are 
defined by 
\begin{align*} 
& \mathcal T[\eta,\beta b]w = -\frac 1 3 (\eta^3w_x)_x + \frac \beta 2[(\eta^2b' 
w)_x - \eta^2b' w_x] + \beta^2\eta(b')^2w, \\ 
& \begin{multlined}[.99\textwidth] 
\mathcal Q[\eta,\beta b]w = \frac{\beta}{2\eta}\left\{(\eta^2 ww_xb')_x + 
(\eta^2b'' w^2)_x - \eta^2[ww_{xx}-w^2_x]b'\right\} \\ 
+ \beta^2b'b''w^2 + \beta^2ww_x(b')^2. 
\end{multlined} 
\end{align*} 

In these variables the bottom topography is given by $z=-\eta_b(x)$, where 
$\eta_b(x) = 1-\beta b(x)$ and $b$ is assumed to be a $C^2_b$ function. Since 
the free surface is at $z=\varepsilon\zeta(x,t)$, cf.\ Figure \ref{fig:1}, the 
{\em water depth} $\eta$ in \eqref{eq:SGN} is given by $\eta=\varepsilon\zeta + 
\eta_b$. 

The assumptions under which \eqref{eq:SGN} is a valid approximation to the 
2d-Euler equations are, cf.\ \cite{LB}, 
\begin{equation}\label{eq:1p1} 
\mu\ll1,\quad \varepsilon=\mathcal O(1),\quad \beta=\mathcal O(1). 
\end{equation} 
It may then be seen that the second pde in \eqref{eq:SGN} is formally an 
$\mathcal O(\mu^2)$ approximation of an analogous expression for the Euler 
equations. (The first pde is exact.) 
It will be also assumed that the bottom never reaches the undisturbed surface 
i.e.\ that $\eta_b(x) = 1-\beta b(x)>0$, for all $x$. It will also be assumed 
that at $t=0$ the water depth $\eta$ is positive. Part of the theory of 
existence-uniqueness of solutions of the initial-value problem for 
\eqref{eq:SGN} is to prove that the data is such that $\eta$ remains positive 
for the duration of existence of solutions. This is what is proved locally in 
time and in some generality in \cite{La}. Also, in the case of the 1d 
\eqref{eq:SGN} as given above, a local in time theory of existence and 
uniqueness of solutions of the ivp with energy methods has been given by Israwi, 
\cite{I}. 

The model \eqref{eq:SGN} has been used in many computational studies of 
long-surface wave propagation over uneven bottoms. We refer, for example, to 
\cite{B}, \cite{CBM2} and \cite{BCLMT} and their references for computations 
with finite differences and finite volume methods, and to \cite{MSMc} for a 
finite element scheme. An error analysis of the Galerkin-finite element method 
in the case of a horizontal bottom (i.e.\ when $\beta=0$), appears in 
\cite{ADM17} in the case of the periodic ivp. 

As was mentioned previously, our aim in this paper is to consider two 
simplifications of \eqref{eq:SGN} that are variable-bottom extensions of 
\eqref{eq:CB}. The derivation of the first of those systems, in addition to 
\eqref{eq:1p1}, is made under the Boussinesq scaling hypothesis that 
$\varepsilon=\mathcal O(\mu)$ (we usually take $\varepsilon=\mu$), and allows 
arbitrary bottom topography, i.e.\ $\beta=\mathcal O(1)$, cf.\ \cite{LB}. When 
we take $\varepsilon=\mu$ in the second pde in \eqref{eq:SGN} and ignore 
$\mathcal O(\mu^2)$ terms (thus retaining the formal accuracy of \eqref{eq:SGN} 
as approximation of the Euler equations), it is not hard to see that 
$\eta=\eta_b+\mathcal O(\mu)$ and that $\frac \mu \eta = \frac{\mu}{\eta_b} + 
\mathcal O(\mu^2)$. Therefore $\frac \mu \eta T[\eta,\beta b]w = 
\frac{\mu}{\eta_b}T[\eta_b,\beta b]w + \mathcal O(\mu^2)$. So, since 
$\mu\varepsilon = \mathcal O(\mu^2)$, if we ignore $\mathcal O(\mu^2)$ terms the 
second pde in \eqref{eq:SGN} becomes 
\[
\left(1+\frac{\mu}{\eta_b}\mathcal T[\eta_b\beta b]\right)u_t + \zeta_x + 
\varepsilon uu_x = 0. 
\] 
Together with the first pde in \eqref{eq:SGN} we obtain therefore a simplified 
system of equations that incorporates the hypothesis $\varepsilon=\mathcal 
O(\mu)$ but allows $\beta=\mathcal O(1)$. This system will be called {\em 
`Classical' Boussinesq system with strongly varying bottom topography} and 
abbreviated as \eqref{eq:CBs}. It is given by the pde's 
\begin{equation}\label{eq:CBs}\tag{CBs} 
\begin{aligned} 
& \zeta_t + (\eta u)_x = 0, \\ 
& \left(1+\frac{\mu}{\eta_b} \mathcal T[\eta_b,\beta b]\right)u_t + \zeta_x + 
\varepsilon uu_x = 0, 
\end{aligned} 
\end{equation} 
where $\eta=\eta_b+\varepsilon\zeta>0$, $\eta_b = 1-\beta b>0$, 
$\varepsilon=\mathcal O(\mu)\ll 1$, and $\mathcal T[\eta_b,\beta b]w$ is given 
by its expression in \eqref{eq:SGN} when we replace $\eta$ by $\eta_b$. 

This system, as a little algebra shows, coincides with the system that was first 
derived from the Euler equations by Peregrine in \cite{P}; it is usually called 
the `Peregrine system' in the literature and has been used widely in practice in 
coastal dynamics computations. We will refer to several computational studies 
with \eqref{eq:CBs} in Section \ref{sec:3} of the present paper. If we now 
assume in \eqref{eq:CBs} following \cite{LB} that $\beta=\mathcal 
O(\varepsilon)$, i.e.\ that the variation of bottom is small and specifically of 
the order $\varepsilon$ of the nonlinear and dispersion terms in \eqref{eq:CBs}, 
we obtain a system that we will call here the {\em `Classical' Boussinesq system 
with weakly varying bottom topography}, \eqref{eq:CBw}. It is straightforward to 
see that if $\beta = \mathcal O(\varepsilon)$ the first equation in 
\eqref{eq:CBs} remains intact and that the second equation, up to $\mathcal 
O(\varepsilon^2)$ terms that we neglect, coincides with the second equation in 
\eqref{eq:CB}. Thus we have the system 
\begin{equation}\label{eq:CBw}\tag{CBw} 
\begin{aligned} 
& \zeta_t + (\eta u)_x = 0, \\ 
& u_t + \zeta_x + \varepsilon uu_x - \frac \mu 3 u_{xxt} = 0, 
\end{aligned} 
\end{equation} 
where of course we still assume that $\varepsilon=\mathcal O(\mu)$, $\mu\ll 1$. 
The dependence on the bottom topography occurs now explicitly (but weakly) 
through the first pde only, since $\eta = \eta_b + \varepsilon\zeta = 1-\beta b 
+ \varepsilon\zeta$ with $\beta=\mathcal O(\varepsilon)$. This system has also 
been used widely in computations in the engineering literature, and coincides 
with the system derived in \cite{C1}. It should be noted that another rigorous 
derivation of the two variable-bottom `Classical' Boussinesq systems and various 
other associated models has been given in \cite{Ch}. 

The theory of existence and uniqueness of solutions, at least locally in time, 
for the ivp for \eqref{eq:CBs} may be easily inferred from the analogous theory 
of \eqref{eq:SGN}, cf.\ e.g.\ \cite{I}, while that of \eqref{eq:CBw} is 
practically the same as the one for \eqref{eq:CB} plus a `source'-type linear 
term of the form $-\beta(bu)_x$ in the left-hand side of the first pde. 

In this paper we will discretize in space ibvp's for the systems \eqref{eq:CBs} 
and \eqref{eq:CBw}, with zero b.c.\ for $u$ at the endpoints of $[0,1]$ and no 
b.c.\ for $\zeta$, by the standard Galerkin-finite element method on a 
quasiuniform mesh and prove $L^2$-error estimates in Section \ref{sec:2} for the 
resulting semidiscretizations. Under certain standard assumptions on the finite 
element spaces we will prove error estimates of the form 
\begin{equation}\label{eq:1p2} 
\|\zeta-\zeta_h\| + \|u-u_h\| \leq C h^{r-1}, 
\end{equation} 
where $\zeta_h$, $u_h$ are the semidiscrete approximations of $\zeta$ and $u$, 
respectively, $h=\max_i h_i$, and $r-1\geq 2$ is the degree of the piecewise 
polynomials in the finite element space. ($\|\cdot\|$ and $\|\cdot\|_1$ denote, 
respectively the $L^2$ and $H^1$ norms of functions on $[0,1]$.) This type of 
error estimate is of the same type as the one proved in \cite{AD2} for the 
analogous ibvp for \eqref{eq:CB} in the case of a quasiuniform mesh. 

In Section \ref{sec:3} we show the results of several numerical experiments that 
we performed with both systems using a fully discrete scheme with the above 
spatial discretization and with temporal discretization effected by the 
classical, 4\textsuperscript{th} order, 4-stage Rugne-Kutta method. The 
resulting schemes are stable under a mild Courant number restriction and highly 
accurate. In Section \ref{subsec:3p1} we check that the schemes also work for 
piecewise linear continuous functions (i.e.\ for $r=2$ and are of optimal order 
in $L^2$ for both $u$ and $\zeta$ in the case of uniform mesh. In Section 
\ref{subsec:3p2} we discuss the application of simple, approximate, absorbing 
boundary conditions for the systems as an alternative to the reflection b.c.\ 
$u=0$ at the endpoints. In Section \ref{subsec:3p3} we perform a series of 
numerical experiments aimed at describing in detail the changes that solitary 
waves undergo when evolving under \eqref{eq:CBs} or \eqref{eq:CBw} in a variety 
of variable-bottom environments. We assess the efficacy of these systems in 
approximating these flows by comparing them with each other and with the 
\eqref{eq:SGN} system and available experimental data. In the Ph.D.\ thesis of 
the first listed author, \cite{GK}, one may find more details on the theory 
behind, and more numerical experiments with these systems, as well as with 
related models of surface water wave propagation over variable bottom. 

In the sequel, we denote, for integer $k\geq 0$, $C^k=C^k[0,1]$ the spaces of 
$k$-times continuously differentiable functions on $[0,1]$ and by $H^k=H^k(0,1)$ 
the usual $L^2$-based Sobolev spaces of functions on $(0,1)$. $\hzr^1$ will 
denote the elements of $H^1$ which vanish at $x=0$ and $x=1$. The inner product 
in $L^2=L^2(0,1)$ will be denoted by $(\cdot,\cdot)$, its norm by $\|\cdot\|$, 
and the norm on $H^k$ by $\|\cdot\|_k$. The norms on $W_\infty^k$ and $L^\infty$ 
on $(0,1)$ are denoted by $\|\cdot\|_{k,\infty}$ and $\|\cdot\|_\infty$, 
respectively. $\mathbb P_r$ are the polynomials of degree at most $r$.

\section{Error analysis of the Galekin semidiscretization} 
\label{sec:2} 

\subsection{The finite element spaces} 
\label{subsect:2p1} 

Let $0\leq x_1<x_2<\ldots<x_{N+1}=1$ be a quasiuniform partition of $[0,1]$ with 
$h:= \max_i(x_{i+1}-x_i)$. For integers $r\geq 2$ and $0\leq k\leq r-2$ we 
consider the finite element space $S_h = \{\phi\in C^k : 
\phi\big|_{[x_i,x_{i+1}]}\in \mathbb P_{r-1}\}$ and $S_{h,0} = \{\phi\in S_h : 
\phi(0)=\phi(1) = 0\}$. It is well known that if $w\in H^r$ there exists 
$\chi\in S_h$ such that 
\begin{equation}\label{eq:2p1} 
\|w-\chi\| + h\|w'-\chi'\| \leq C h^r \|w\|_r 
\end{equation} 
for some constant $C$ independent of $h$ and $w$, and that a similar property 
holds in $S_{h,0}$ provided $w\in H^r\cap H^1_0$. In addition, if $\opP$ is the 
$L^2$ prejection operator onto $S_h$, then it holds, cf.\ \cite{DDW}, that 
\begin{subequations} \label{eq:2p2}
\begin{align} 
& \|\opP v\|_\infty \leq C \|v\|_\infty,\quad \forall v\in L^\infty, 
\label{eq:2p2a} \\ 
& \|\opP v-v\|_\infty \leq C h^r \|v\|_{r,\infty},\quad \forall v\in C^r. 
\label{eq:2p2b} 
\end{align} 
\end{subequations} 
Due to the quasiuniformity of the mesh, the inverse inequalities 
\begin{equation}\label{eq:2p3} 
\|\chi\|_1 \leq C h^{-1} \|\chi\|,\quad \|\chi\|_\infty\leq C h^{-1/2}\|\chi\| 
\end{equation} 
are valid for $\chi\in S_h$ (or $\chi\in S_{h,0}$).

\subsection{Semidiscretization in the case of a strongly varying bottom}
\label{subsec:2p2} 

Using the notation of the Introduction we consider the following 
initial-boundary-value problem (ibvp) for \eqref{eq:CBs}. For $T>0$ we seek 
$\zeta=\zeta(x,t)$, $u=u(x,t)$, for $(x,t)\in [0,1]\times [0,T]$, such that 
\begin{equation}\label{eq:2p4} 
\begin{aligned} 
& \begin{aligned} 
& \zeta_t + (\eta u)_x = 0, \\ 
& \left(1 + \frac{\mu}{\eta_b}\mathcal T[\eta_b,\beta b]\right)u_t + \zeta_x + 
\varepsilon uu_x = 0, 
\end{aligned} && 0\leq x\leq1,\ \ 0\leq t\leq T, \\ 
& \zeta(x,0) = \zeta_0(x),\ \ u(x,0) = u_0(x), && 0\leq x\leq 1, \\ 
& u(0,t) = u(1,t) = 0, && 0\leq t\leq T, 
\end{aligned} 
\end{equation} 
where $\zeta_0$, $u_0$ are given functions of $[0,1]$ and 
\[ 
\eta=\varepsilon\zeta + \eta_b>0,\quad \eta_b(x) = 1-\beta b(x)>0, 
\] 
$\varepsilon$, $\mu$, $\beta$, are positive constants with $\varepsilon=\mathcal 
O(\mu)$, $\mu\ll 1$, $\beta=\mathcal O(1)$, and $b\in C^2[0,1]$. The operator 
$\mathcal T[\eta_b,\beta b]$ is defined as in Section \ref{sec:1} by 
\[ 
\mathcal T[\eta_b,\beta b]w = -\frac 1 3 (\eta_b^3w_x)_x + \frac \beta 2 
[(\eta_b^2b'w)_x - \eta_b^2b'w_x] + \beta^2\eta_b(b')^2w. 
\]
All the variables above are nondimensional and scaled. We will assume that the 
ibvp \eqref{eq:2p4} has a unique solution that is smooth enough for the purposes 
of the error estimates to follow. Taking into account that 
\[ 
\mathcal T[\eta_b,\beta b]w = -\frac 1 3 (\eta_b^3w_x)_x + \frac \beta 2 
(\eta_b^2b')'w + \beta^2\eta_b(b')^2w, 
\] 
and that $\eta_b'=-\beta b'$, we have 
\begin{equation}\label{eq:2p5} 
\mathcal T[\eta_b,\beta b]w = -\frac 1 3 (\eta_b^3w_x)_x - \frac 1 2 
\eta_b^2\eta_b''w. 
\end{equation} 

Using in first pde of \eqref{eq:2p4} the definition of $\eta$, multiplying the 
second pde by $\eta_b$, and taking into account \eqref{eq:2p5}, we rewrite the 
ibvp \eqref{eq:2p4} in the form 
\begin{equation}\label{eq:2p6} 
\begin{aligned} 
& \begin{aligned} 
& \zeta_t + \varepsilon(\zeta u)_x + (\eta_bu)_x = 0, \\ 
& \left(\eta_b - \frac \mu 2 \eta_b^2\eta_b''\right)u_t - \frac \mu 3 
(\eta_b^3u_{tx})_x + \eta_b\zeta_x + \varepsilon\eta_buu_x = 0, 
\end{aligned}\quad (x,t)\in [0,1]\times[0,T], \\ 
& \zeta(x,0)=\zeta_0(x),\ \ u(x,0)=u_0(x),\quad x\in [0,1], \\ 
& u(0,t)=u(1,t)=0,\quad t\in [0,T]. 
\end{aligned} 
\end{equation} 
We assume that there are positive constants $c_1$ and $c_2$ such that 
\begin{subequations} \label{eq:2p7} 
\begin{align} 
\eta_b(x) &\geq c_1, \label{eq:2p7a} \\ 
\eta_b(x) - \frac \mu 2 \eta_b^2(x)\eta_b''(x) &\geq c_2, \label{eq:2p7b} 
\end{align} 
\end{subequations} 
for all $x\in [0,1]$. Since $\eta_b$ and its derivatives are $\mathcal O(1)$, 
\eqref{eq:2p7b} holds for $\mu$ sufficiently small. We also consider the 
bilinear form $A: H_0^1\times H_0^1 \to \mathbb R$ defined by 
\begin{equation}\label{eq:2p8} 
A(v,w) = \left((\eta_b - \frac \mu 2 \eta_b^2\eta_b'')v,w\right) + \frac \mu 3 
(\eta_b^3v',w'), 
\end{equation} 
which is symmetric, bounded on $H^1\times H^1$, and, because of \eqref{eq:2p7}, 
coercive, with 
\begin{equation}\label{eq:2p9} 
A(v,v) \geq c_2\|v\|^2 + \frac{\mu c_1^3}{3}\|v'\|^2 \geq c_\mu\|v\|_1^2,\quad 
\forall v\in H^1, 
\end{equation} 
where $c_\mu := \min(c_2,\mu c_1^3/3)$. Consider now a weighted $H^1$ 
(`elliptic') projection associated with the bilinear form \eqref{eq:2p9} as the 
map $\opR_h : \hzr^1\to S_{h,0}$ defined by 
\begin{equation}\label{eq:2p10} 
A(\opR_h v,\chi) = A(v,\chi),\quad \forall \chi\in S_{h,0}, 
\end{equation} 
for which, cf.\ e.g.\ \cite{DDW}, it holds that 
\begin{gather} 
\|\opR_hv-v\| + h\|\opR_hv-v\|_1 \leq C h^r \|v\|_r,\quad \text{if}\ \ v\in 
H^r\cap H_0^1, \label{eq:2p11} \\ 
\|\opR_hv-v\|_\infty \leq C h^r \|v\|_{r,\infty}, \quad \text{if}\ \ v\in 
W^{r,\infty}\cap H_0^1. \label{eq:2p12} 
\end{gather} 

We now define the standard Galerkin finite element semidiscretization of the 
ibvp \eqref{eq:2p6}. We seek $\zeta_h : [0,T]\to S_h$, $u_h : [0,T]\to S_{h,0}$ 
such that 
\begin{align} 
& (\zeta_{ht},\phi) + \varepsilon\bigl((\zeta_hu_h,\phi)_x,\phi\bigr) + 
\bigl((\eta_bu_h)_x,\phi\bigr) = 0,\quad  \forall \phi\in S_h, \quad  
\smash{\raisebox{-.5\baselineskip}{$\displaystyle 0\leq t\leq T,$}} 
\label{eq:2p13} \\ 
& A(u_{ht},\chi) + (\eta_b\zeta_{hx},\chi) + 
\varepsilon(\eta_bu_hu_{hx},\chi)=0,\quad \forall\chi\in S_{h,0}, 
\label{eq:2p14} 
\end{align}
with initial conditions 
\begin{equation}\label{eq:2p15} 
\zeta_h(0)=\opP\zeta_0,\quad u_h(0)=\opR_h u_0. 
\end{equation} 
The ode ivp given by \eqref{eq:2p13}--\eqref{eq:2p15} has a unique solution 
locally in time. As part of Theorem \ref{thm:2p1} below we will prove that for 
sufficiently small $h$, its solution may be extended up to $t=T$. 

\begin{theorem}\label{thm:2p1} 
Suppose that the solution $(\zeta,u)$ of \eqref{eq:2p6} is sufficiently smooth 
and that the conditions \eqref{eq:2p7} hold. Then, if $h$ sufficiently small, 
there exists a constant $C$ independent of $h$ such that the semidiscrete 
problem \eqref{eq:2p13}--\eqref{eq:2p15} has a unique solution $(\zeta_h,u_h)$ 
for $0\leq t\leq T$, that satisfies 
\begin{equation}\label{eq:2p16}  
\max_{0\leq t\leq T}\left(\|\zeta(t)-\zeta_h(t)\| + \|u(t)-u_h(t)\|_1\right) 
\leq C h^{r-1}. 
\end{equation} 
\end{theorem} 

\begin{proof} 
Let $\rho = \zeta-\opP\zeta$, $\theta=\opP\zeta-\zeta_h$, $\sigma=u-\opR_h u$, 
$\xi=\opR_h u-u_h$. From \eqref{eq:2p6} and \eqref{eq:2p13}--\eqref{eq:2p15} we 
get 
\begin{align} 
& (\theta_t,\phi) + \varepsilon\bigl((\zeta u-\zeta_hu_h)_x,\phi\bigr) + 
\big((\eta_b\sigma+\eta_b\xi)_x,\phi)=0,\quad \forall\phi\in S_h, 
\label{eq:2p17} \\ 
& A(\xi_t,\chi) + \bigl(\eta_b(\rho_x+\theta_x),\chi\bigr) + 
\varepsilon\bigl(\eta_b(uu_x-u_hu_{hx}),\chi\bigr) = 0,\quad \forall\chi\in 
S_{h,0}, \label{eq:2p18} 
\end{align} 
that are valid while the semidiscrete problem has a unique solution. For the 
nonlinear terms we have 
\begin{align*} 
& \zeta u - \zeta_hu_h = \zeta(\sigma+\xi) + u(\rho+\theta) - 
(\rho+\theta)(\sigma+\xi), \\ 
& uu_x - u_hu_{hx} = (u\sigma)_x + (u\xi)_x - (\sigma\xi)_x - \sigma\sigma_x 
-\xi\xi_x. 
\end{align*} 
Let now $t_h\in (0,T]$ be the maximal temporal instance for which the solution 
of \eqref{eq:2p6} exists and it holds that $\|\theta(t)\|_\infty + 
\|\xi(t)\|_\infty \leq 1$, for $t\leq t_h$. Putting $\phi=\theta$ in 
\eqref{eq:2p17}, using \eqref{eq:2p1}, \eqref{eq:2p2b}, \eqref{eq:2p11}, 
\eqref{eq:2p12}, \eqref{eq:2p3}, and integrating by parts we have for $t\leq 
t_h$ 
\begin{equation}\label{eq:2p19} 
\begin{aligned} 
\tfrac 1 2 \tfrac{\mathrm d}{\mathrm dt}\|\theta\|^2 = & 
-\varepsilon\big((\zeta\sigma)_x,\theta\big) - 
\varepsilon\big((\zeta\xi)_x,\theta\big) - \varepsilon\big((u\rho)_x,\theta\big) 
- \varepsilon\big((u\theta)_x,\theta\big) + 
\varepsilon\big((\rho\sigma)_x,\theta\big) \\ 
& + \varepsilon\big((\rho\xi)_x,\theta) + 
\varepsilon\big((\sigma\theta)_x,\theta\big) + 
\varepsilon\big((\theta\xi)_x,\theta\big) - \big((\eta_b\sigma)_x,\theta\big) - 
\big((\eta_b\xi)_x,\theta\big) \\ 
\leq &\ \varepsilon(\|\zeta_x\|_\infty\|\sigma\| + 
\|\zeta\|_\infty\|\sigma_x\|)\|\theta\| + \varepsilon(\|\zeta_x\|_\infty\|\xi\| 
+ \|\zeta\|_\infty\|\xi_x\|)\|\theta\| \\ 
& + \varepsilon(\|u_x\|_\infty\|\rho\| + \|u\|_\infty\|\rho_x\|)\|\theta\| + 
\tfrac \varepsilon 2 \|u_x\|_\infty\|\theta\|^2 \\ 
& + \varepsilon(\|\sigma\|_\infty\|\rho_x\| + 
\|\rho\|_\infty\|\sigma_x\|)\|\theta\| + \varepsilon(\|\xi\|_\infty\|\rho_x\| + 
\|\rho\|_\infty\|\xi_x\|)\|\theta\| \\ 
& + \tfrac \varepsilon 2 \|\theta\|_\infty\|\sigma_x\|\|\theta\| + \tfrac 
\varepsilon 2 \|\theta\|_\infty\|\xi_x\|\|\theta\| + 
(\|\eta'_\beta\|_\infty\|\sigma\| + \|\eta_b\|_\infty\|\sigma_x\|)\|\theta\| \\ 
& + (\|\eta'_\beta\|_\infty\|\xi\| + \|\eta_b\|_\infty\|\xi_x\|)\|\theta\| \\ 
\leq &\ C(h^{r-1} + \|\xi\|_1 + \|\theta\|)\|\theta\|, 
\end{aligned} 
\end{equation} 
for some constant $C$ independent of $h$. 

In addition, with $\chi=\xi$ in \eqref{eq:2p18} we obtain for $t\leq t_h$ 
\begin{align*} 
\tfrac 1 2 \tfrac{\mathrm d}{\mathrm dt} A(\xi,\xi) =& 
-(\eta_b\rho_x+\eta_b\theta_x,\xi) - \varepsilon(\eta_b(u\sigma)_x,\xi) - 
\varepsilon(\eta_b(u\xi)_x,\xi) + \varepsilon(\eta_b(u\xi)_x,\xi) \\ 
& + \varepsilon(\eta_b\sigma\sigma_x,\xi) + \varepsilon(\eta_b\xi\xi_x,\xi) \\ 
 =& (\rho,\eta'_b\xi+\eta_b\xi_x) + (\theta,\eta_b'\xi+\eta_b\xi_x) + 
\varepsilon(u\sigma,\eta_b'\xi+\eta_b\xi_x) - \varepsilon(\eta_b(u\xi)_x,\xi) \\ 
 & -\varepsilon(\sigma\xi,\eta_b'\xi+\eta_b\xi_x) - \tfrac \varepsilon 
2(\sigma^2,\eta_b'\xi+\eta_b\xi_x) - \tfrac \varepsilon 3 \eta_b'\xi^2,\xi). 
\end{align*} 
With estimates analogous to those used in \eqref{eq:2p19} we get 
\begin{equation}\label{eq:2p20} 
\begin{aligned} 
\tfrac 1 2 \tfrac{\mathrm d}{\mathrm dt} A(\xi,\xi) \leq &\ 
(\|\eta_b'\|_\infty\|\xi\| + \|\eta_b\|_\infty\|\xi_x\|)(\|\rho\|+\|\theta\|) + 
\varepsilon\|u\eta_b'\|_\infty\|\sigma\|\|\xi\| \\ 
& + \varepsilon\|u\eta_b\|_\infty\|\sigma\|\|\xi_x\| + 
\varepsilon\|\eta_bu_x\|_\infty\|\xi\|^2 + 
\varepsilon\|\eta_bu\|_\infty\|\xi_x\|\|\xi\| \\ 
& + \varepsilon\|\sigma\eta_b'\|_\infty\|\xi\|^2 + 
\varepsilon\|\sigma\eta_b\|_\infty\|\xi\|\|\xi_x\| + \tfrac \varepsilon 2 
\|\sigma\eta_b'\|_\infty\|\sigma\|\|\xi\| \\ 
& + \tfrac \varepsilon 2 \|\sigma\eta_b\|_\infty\|\sigma\|\|\xi_x\| + \tfrac 
\varepsilon 3 \|\eta_b'\|_\infty\|\xi\|_\infty\|\xi\|^2 \\ 
\leq&\ C(h^r + \|\xi\|_1 + \|\theta\|)\|\xi\|_1, 
\end{aligned} 
\end{equation} 
where $C$ is independent of $h$. From \eqref{eq:2p19} and \eqref{eq:2p20} we see 
that 
\[
\tfrac{\mathrm d}{\mathrm dt}\big(\|\theta\|^2 + A(\xi,\xi)\big) \leq C_1 
h^{2r-2} + C_2\left(\|\theta\|^2+\|\xi\|_1^2\right), 
\] 
where $C_1$, $C_2$ are independent of $h$. From this inequality and 
\eqref{eq:2p9} it follows that 
\[ 
\tfrac{\mathrm d}{\mathrm dt}\big(\|\theta\|^2+A(\xi,\xi)\big) \leq C_1h^{2r-2} 
+ C_\mu\big(\|\theta\|^2+A(\xi,\xi)\big), 
\] 
for $t\leq t_h$, where $C_\mu=C_2\max(1,1/c_\mu)$. Using Gronwall's lemma in the 
above we obtain for $t\leq t_h$, 
\[ 
\|\theta(t)\|^2 + A(\xi(t),\xi(t)) \leq \mathrm e^{C_\mu T}\big(\|\theta(0)\|^2 
+ A(\xi(0),\xi(0)\big) + \tfrac{C_1}{C_\mu}\mathrm e^{C_\mu T}h^{2r-2}, 
\]
from which, in view of \eqref{eq:2p9} and since $\theta(0)= \xi(0) = 0$, we see 
that 
\begin{equation}\label{eq:2p21} 
\|\theta(t)\| + \|\xi(t)\|_1 \leq \left(\frac{2C_1}{C_\mu\widetilde 
C_\mu}\mathrm e^{C_\mu T}\right)^{1/2}h^{r-1}, 
\end{equation} 
for $t\leq t_h$, where $\widetilde C_\mu=\min(1,c_\mu)$. Now, since 
\eqref{eq:2p3} gives $\|\theta\|_\infty\leq C h^{-1/2}\|\theta\|$ and 
$\|\xi\|_\infty\leq \|\xi\|_1$, if $h$ is taken sufficiently small, we have that 
$\|\theta\|_\infty + \|\xi\|_\infty<1$ for $0\leq t\leq t_h$, and therefore we 
may take $t_h=T$. The result follows from \eqref{eq:2p20} and the approximation 
properties of the finite element spaces. 
\end{proof} 

As suggested by numerical experiments for the \eqref{eq:CB} on a {\em horizontal 
bottom}, shown in \cite{AD2}, the convergence rates in the error estimate 
\eqref{eq:2p16} are sharp in the case of a horizontal bottom; they are sharp in 
the case of variable-bottom models as well. The $H^1$ convergence rate of the 
error of $u_h$ is optimal, while the $L^2$ rate for $\eta_h$ suboptimal, as 
expected, since the first pde in \eqref{eq:2p4} is of hyperbolic type and we are 
using the standard Galerkin method on a nonuniform mesh. (For $r=2$ the 
numerical experiments in \cite{AD2} also suggest the improved estimate 
$\|u-u_h\|=\mathcal O(h^2)$.) In the case of {\em uniform mesh}, better results 
were proved in \cite{AD2} in the case of horizontal bottom. The numerical 
experiments in Section \ref{sec:3} of the paper at hand suggest that such 
improved rates of convergence for uniform mesh persist in the presence of a 
variable bottom as well.

\subsection{Semidiscretization in the case of a weakly varying bottom} 
\label{subsec:2p3} 

In the case of a weakly varying bottom, following the remarks in Section 
\ref{sec:1}, we consider the following ibvp for the system \eqref{eq:CBw}. For 
$T>0$ we seek $\zeta=\zeta(x,t)$, $u=u(x,t)$, for $(x,t)\in [0,1]\times [0,T]$, 
such that 
\begin{equation} \label{eq:2p22} 
\begin{aligned} 
& \begin{aligned} 
& \zeta_t + (\eta u)_x = 0, \\ 
& u_t + \zeta_x + \varepsilon uu_x - \frac \mu 3 u_{xxt} = 0, 
\end{aligned}\quad 0\leq x\leq 1,\ \ 0\leq t\leq T, \\ 
& \zeta(x,0) = \zeta_0(x),\ \ u(x,0)=u_0(x),\quad 0\leq x\leq 1, \\ 
& u(0,t)=u(1,t),\quad 0\leq t\leq T, 
\end{aligned} 
\end{equation} 
where 
\[ 
\eta=\varepsilon\zeta + \eta_b > 0,\quad \eta_b(x)=1-\beta b(x)>0, 
\] 
and $\varepsilon$, $\mu$, $\beta$, are positive constants with 
$\varepsilon=\mathcal O(\mu)$, $\beta=\mathcal O(\mu)$, $\mu\ll 1$, and 
$b=C^2[0,1]$. All the variables above are nondimensional and scaled. We assume 
that \eqref{eq:2p22} has a unique solution, smooth enough for the purposes of 
the error estimate below. 

Let $a : H_0^1\times H_0^1 \to\mathbb R$ denote the weighted $H^1$-inner product 
defined by $a(v,w) = (v,w) + \tfrac \mu 3 (v',w')$ and consider the weighted 
$H^1$ (`elliptic') projection associated with $a(\cdot,\cdot)$, defined as the 
map $\widetilde \opR_h : \hzr^1\to S_{h,0}$ such that 
\begin{equation}\label{eq:2p23} 
a(\widetilde \opR_hv,\chi) = a(v,\chi),\quad \forall \chi\in S_h. 
\end{equation} 
Obviously, $\widetilde \opR_h$ satisfies the properties \eqref{eq:2p11} and 
\eqref{eq:2p12}. 

The standard Galerkin finite element semidiscretization of the ibvp 
\eqref{eq:2p22} is the following. We seek $\zeta_h : [0,T]\to S_h$, $u_h : [0,T] 
\to S_{h,0}$, such that 
\begin{align} 
& (\zeta_{ht},\phi) + \varepsilon\big((\zeta_hu_h)_x,\phi\big) + 
\big((\eta_bu_h)_x,\phi\big) = 0,\quad \forall\phi\in S_h, \label{eq:2p24} \\ 
& a(u_{ht},\chi) + (\zeta_{hx},\chi) + \varepsilon(u_hu_{hx},\chi) = 0,\quad 
\forall\chi\in S_{h,0}, \label{eq:2p25} 
\end{align} 
with initial conditions 
\begin{equation}\label{eq:2p26} 
\zeta_h(0) = \opP\zeta_0,\quad u_h(0) = \widetilde \opR_hu_0. 
\end{equation} 
In analogy with Theorem \ref{thm:2p1}, the following error estimate holds for 
the semidiscrete scheme \eqref{eq:2p24}--\eqref{eq:2p26}. (We omit the proof 
since it is very similar to that of Theorem \ref{thm:2p1}, {\em mutatis 
mutandis}.) 
\begin{theorem}\label{thm:2p2} 
Suppose that the solution $(\zeta,u)$ of \eqref{eq:2p22} is sufficiently smooth. 
Then, if $h$ is sufficiently small, there exists a constant $C$ independent of 
$h$ such that the semidiscrete problem \eqref{eq:2p24}--\eqref{eq:2p26} has a 
unique solution $(\zeta_h,u_h)$ for $0\leq t\leq T$, that satisfies 
\begin{equation}\label{eq:2p27} 
\max_{0\leq t\leq T}\left(\|\zeta(t)-\zeta_h(t)\| + \|u(t)-u_h(t)\|_1\right) 
\leq C h^{r-1}. 
\end{equation} 
\end{theorem}

\section{Numerical experiments} 
\label{sec:3} 

In this section we present results of numerical experiments that we performed 
using the two models \eqref{eq:CBs} and \eqref{eq:CBw} of the classical Bousinesq 
system with variable bottom. We discretized the two systems in space using the 
Galerkin finite element method analyzed in the previous section. For the 
temporal discretization we used the `classical', explicit, 4-stage, 
4\textsuperscript{th} order Runge-Kutta scheme (RK4). The convergence of this 
fully discrete scheme was analyzed, in the case of the ibvp for the 
\eqref{eq:CB} with horizontal bottom and $u=0$ at the endpoints in \cite{AD2}, 
where it was shown that under a Courant number stability restriction of the form 
$\tfrac k h \leq \alpha$ the scheme is stable, is fourth-order accurate in time, 
and preserves the spatial order of convergence of the semidiscrete problem; here 
$k$ denotes the (uniform) time step.

\subsection{Convergence rates} 
\label{subsec:3p1} 

The spatial convergence rates proved in Theorems \ref{thm:2p1} and \ref{thm:2p2} 
in the case of a general quasiuniform mesh are sharp as is suggested by 
numerical experiments (not shown here). In the case of a uniform spatial mesh 
better convergence rates may be achieved. This was proved in \cite{AD2} for the 
\eqref{eq:CB} (horizontal bottom and $u=0$ at the endpoints of the spatial 
interval) in the case of piecewise linear continuous functions ($r=2$) and cubic 
splines ($r=4$). The numerical results to be presented in the sequel suggest 
that the improved rates persist in the case of a variable bottom as well for 
both CB models. (We do not show the optimal-order results for the piecewise 
linear case ($r=2$) but concentrate instead in the case of cubic splines 
($r=4$).) 

The exact solution of the test problem used for the error rate computations is 
\linebreak[4] 
$\zeta(x,t) = \mathrm e^{2t}(\cos(\pi x) + x + 2)$, $u(x,t) = \mathrm 
e^{xt}(\sin(\pi x) + x^3 - x^2)$ for $(x,t)\in [0,1]\times [0,1/4]$; the bottom 
topography was given by the function $\eta_b(x) = 1-\beta\sin\pi x$. The scaling 
parameters (not important for the convergence rate computations) were taken as 
$\varepsilon=1$, $\mu=1/10$, $\beta=1/10$. Appropriate right-hand sides and 
initial conditions were found from the above data. We solved numerically the 
ibvp's \eqref{eq:2p6} and \eqref{eq:2p22} with the above exact solution and 
bottom profile using the spatial discretizations 
\eqref{eq:2p13}--\eqref{eq:2p15} and \eqref{eq:2p24}--\eqref{eq:2p26}, 
respectively, with cubic splines with uniform mesh of meshlength $h=1/N$. The 
temporal discretization was effected by the RK4 scheme with stability 
restriction $\tfrac k h \leq \tfrac 1 4$; the resulting time steps were small 
enough so that the temporal errors were much smaller than the spatial ones. We 
used 3-point Gauss quadrature to evaluate the finite element integrals in every 
mesh interval. (Since we wished to obtain detailed information about the spatial 
convergence rates, we computed throughout with quadruple precision and evaluated 
he $L^2$-errors using 5-point Gauss quadrature and the $L^\infty$ errors by 
taking the maximum value of the error on all these quadrature points.) 

In Table \ref{tab:1} we show the $L^2$, $L^\infty$, and $H^1$ (seminorm) spatial 
errors and convergence rates in the case of the \eqref{eq:CBw} model. The 
numerical results suggest strongly that the $L^2$ rates for $\zeta$ and $u$ are 
equal to $3.5$ and $4$, respectively, the $L^\infty$ rates 
\begin{table}[htbp] 
\setlength{\tabcolsep}{5pt} 
\centering 
\begin{subtable}{\textwidth}\centering\tt
\begin{tabular}[b]{|c|cc|cc|cc|} \hline
$N$ & $L_2$ {\rm error} & {\rm rate} & $L_\infty$ {\rm error} & {\rm rate} & 
$H_1$ \!\!{\rm semi-nrm} & {\rm rate} \\ \hline
 128 & 6.0526e-09 & -      & 5.6333e-08 & -      & 3.5964e-06 & -    	\\ \hline
 256 & 5.3006e-10 & 3.5133 & 6.9294e-09 & 3.0232 & 6.2857e-07 & 2.5164	\\ \hline
 512 & 4.6605e-11 & 3.5076 & 8.5918e-10 & 3.0117 & 1.1046e-07 & 2.5086	\\ \hline
1024 & 4.1074e-12 & 3.5042 & 1.0696e-10 & 3.0059 & 1.9466e-08 & 2.5045	\\ \hline
2048 & 3.6250e-13 & 3.5022 & 1.3343e-11 & 3.0029 & 3.4357e-09 & 2.5023	\\ \hline
4096 & 3.2016e-14 & 3.5011 & 1.6662e-12 & 3.0015 & 6.0686e-10 & 2.5012	\\ \hline
8192 & 2.8287e-15 & 3.5006 & 2.0816e-13 & 3.0007 & 1.0724e-10 & 2.5006	\\ \hline
\end{tabular} 
\subcaption{$\zeta$ \label{tab:1a}}  
\end{subtable}
\medskip\par 
\begin{subtable}{\textwidth}\centering\tt
\begin{tabular}[b]{|c|cc|cc|cc|}\hline
$N$ & $L_2$ {\rm error} & {\rm rate} & $L_\infty$ {\rm error} & {\rm rate} & 
$H_1$ \!\!{\rm semi-nrm} & {\rm rate} \\ \hline
 128 & 2.9812e-10 & -      & 6.0043e-10 & -      & 2.4080e-07 & -   	\\ \hline
 256 & 1.8618e-11 & 4.0011 & 3.7468e-11 & 4.0023 & 3.0098e-08 & 3.0001	\\ \hline
 512 & 1.1632e-12 & 4.0005 & 2.3407e-12 & 4.0006 & 3.7623e-09 & 3.0000	\\ \hline
1024 & 7.2689e-14 & 4.0002 & 1.4628e-13 & 4.0002 & 4.7029e-10 & 3.0000	\\ \hline
2048 & 4.5427e-15 & 4.0001 & 9.1419e-15 & 4.0001 & 5.8786e-11 & 3.0000	\\ \hline
4096 & 2.8391e-16 & 4.0001 & 5.7136e-16 & 4.0000 & 7.3483e-12 & 3.0000	\\ \hline
8192 & 1.7744e-17 & 4.0000 & 3.5710e-17 & 4.0000 & 9.1853e-13 & 3.0000	\\ \hline
\end{tabular} 
\subcaption{$u$ \label{tab:1b}}  
\end{subtable} 
\vspace{-2ex} 
\caption{Spatial errors and rates of convergence, $t=1/4$, \eqref{eq:CBw}, cubic 
splines on uniform mesh, $h=1/N$, \subref{tab:1a}: $\zeta$, \subref{tab:1b}: 
$u$. \label{tab:1}}
\end{table}
equal to $3$ and $4$, while the $H^1$ ones $2.5$ and $3$, respectively. The same 
rates are observed (cf.\ Table \ref{tab:2}) in the numerical integration by the 
same method of the analogous ibvp for the \eqref{eq:CBs} model.  

As a remark of theoretical interest we point out that in the case of the 
analogous ibvp 
\begin{table}[htbp] 
\setlength{\tabcolsep}{5pt} 
\centering 
\begin{subtable}{\textwidth}\centering\tt 
\begin{tabular}[b]{|c|cc|cc|cc|}\hline
$N$ & $L_2$ {\rm error} & {\rm rate} & $L_\infty$ {\rm error} & {\rm rate} & 
$H_1$ \!\!{\rm semi-nrm} & {\rm rate} \\ \hline
 128 & 6.0165e-09 & -      & 5.5101e-08 & -      & 3.5538e-06 & -     	\\ \hline
 256 & 5.2848e-10 & 3.5090 & 6.8528e-09 & 3.0073 & 6.2481e-07 & 2.5079	\\ \hline
 512 & 4.6535e-11 & 3.5054 & 8.5440e-10 & 3.0037 & 1.1012e-07 & 2.5043	\\ \hline
1024 & 4.1044e-12 & 3.5031 & 1.0666e-10 & 3.0019 & 1.9436e-08 & 2.5023	\\ \hline
2048 & 3.6236e-13 & 3.5017 & 1.3324e-11 & 3.0009 & 3.4331e-09 & 2.5012	\\ \hline
4096 & 3.2010e-14 & 3.5009 & 1.6650e-12 & 3.0005 & 6.0663e-10 & 2.5006	\\ \hline
8192 & 2.8284e-15 & 3.5004 & 2.0809e-13 & 3.0002 & 1.0721e-10 & 2.5003	\\ \hline
\end{tabular}
\subcaption{$\zeta$ \label{tab:2a}}  
\end{subtable}
\medskip\par 
\begin{subtable}{\textwidth}\centering\tt 
\begin{tabular}[b]{|c|cc|cc|cc|}\hline
$N$ & $L_2$ {\rm error} & {\rm rate} & $L_\infty$ {\rm error} & {\rm rate} & 
$H_1$ \!\!{\rm semi-nrm} & {\rm rate} \\ \hline
 128 & 2.9818e-10 & -      & 6.0086e-10 & -      & 2.4081e-07 & -     	\\ \hline
 256 & 1.8621e-11 & 4.0012 & 3.7476e-11 & 4.0030 & 3.0099e-08 & 3.0001	\\ \hline
 512 & 1.1634e-12 & 4.0005 & 2.3411e-12 & 4.0007 & 3.7623e-09 & 3.0000	\\ \hline
1024 & 7.2699e-14 & 4.0002 & 1.4630e-13 & 4.0002 & 4.7029e-10 & 3.0000	\\ \hline
2048 & 4.5433e-15 & 4.0001 & 9.1432e-15 & 4.0001 & 5.8786e-11 & 3.0000	\\ \hline
4096 & 2.8395e-16 & 4.0001 & 5.7144e-16 & 4.0000 & 7.3483e-12 & 3.0000	\\ \hline
8192 & 1.7746e-17 & 4.0000 & 3.5714e-17 & 4.0000 & 9.1853e-13 & 3.0000	\\ \hline
\end{tabular} 
\subcaption{$u$ \label{tab:2b}}  
\end{subtable} 
\vspace{-2ex} 
\caption{Spatial errors and rates of convergence, $t=1/4$, \eqref{eq:CBs}, cubic 
splines on uniform mesh, $h=1/N$, \subref{tab:2a}: $\zeta$, \subref{tab:2b}: 
$u$. \label{tab:2}}
\end{table}
for \eqref{eq:CB} on a horizontal bottom two of the authors proved in \cite{AD2} 
$L^2$ error estimates $\|\zeta-\zeta_h\|\leq C h^{3.5} \sqrt{|\ln h|}$, 
$\|u-u_h\|\leq C h^4 \sqrt{|\ln h|}$, for the semidiscretizastion with cubic 
splines on a uniform mesh. The increased accuracy of our present code affords 
investigating computationally if the logarithmic factors are actually present in 
these estimates. To this end we considered the ibvp for \eqref{eq:CB} with the 
exact solution given previously, but now in the case of the horizontal bottom 
$\eta_b=1$, and found that the rates $\frac{\|\zeta-\zeta_h\|}{h^{3.5}}$ 
stabilized to the value 0.124 (the values of $h$ used were less than $1/1024$ 
and the errors $\|\zeta-\zeta_h\|$ were of $\mathcal O(10^{-12})$ or smaller,) 
while the ratio $\frac{\|\zeta-\zeta_h\|}{h^{3.5}\sqrt{|\ln h|}}$ did not 
stabilize for the same range of $h$'s. Similar observations were made for the 
$u$ component of the error. Therefore these increased accuracy experiments 
suggest that the error estimates in \cite{AD2} are not sharp.

\subsection{Approximate absorbing boundary conditions} 
\label{subsec:3p2} 

In the case of the shallow water \eqref{eq:SW} equations on a horizontal bottom, 
obtained if we set $\mu=0$ in the \eqref{eq:CB} system, i.e.\ for the equations 
\begin{equation}\label{eq:SW}\tag{SW} 
\begin{aligned} 
& \zeta_t + u_x + \varepsilon(\zeta u)_x = 0, \\ 
& u_t + \zeta_x + \varepsilon uu_x = 0, 
\end{aligned}
\end{equation} 
(written here in nondimensional, scaled variables, and where it is assumed that 
$1+\varepsilon\zeta>0$), it is well known that using Riemann invariants and the 
theory of characteristics. \cite{W}, one may derive transparent, {\em 
characteristic boundary conditions} at the endpoints of a finite spatial 
interval, say $[0,1]$. These boundary conditions allow an initial pulse that is 
generated in the interior of $(0,1)$ and travels in both directions to exit the 
interval cleanly. In the case of a subcritical flow, which will be of interest 
here, i.e.\ when the solution of \eqref{eq:SW} satisfies 
$u^2<(1+\varepsilon\zeta)/\varepsilon^2$, the characteristic boundary conditions 
are of the form 
\begin{equation}\label{eq:3p1} 
\begin{aligned} 
& \varepsilon u(0,t) + 2\sqrt{1+\varepsilon\zeta(0,t)} = \varepsilon u_0 + 
2\sqrt{1+\varepsilon\zeta_0}, \\ 
& \varepsilon u(1,t) - 2\sqrt{1+\varepsilon\zeta(1,t)} = \varepsilon u_0 - 
2\sqrt{1+\varepsilon\zeta_0}. 
\end{aligned} 
\end{equation} 
Here it is assumed that outside the interval $[0,1]$ the flow is uniform and 
satisfies $\zeta(x,t)=\zeta_0$, $u(x,t)=u_0$, where $\eta_0$, $u_0$ are 
constants such that $u_0^2<(1+\varepsilon\zeta_0)/\varepsilon^2$. In addition, 
the initial conditions $\zeta(x,0)$, $u(x,0)$, of \eqref{eq:SW} should satisfy 
the subcriticality conditions and be compatible at $x=0$ and $x=1$ with the 
uniform flow outside $[0,1]$. In \cite{AD} two of the authors analyzed the space 
discretization of \eqref{eq:SW} with characteristic boundary conditions (both in 
the subcritical and supercritical case) using Galerkin finite element methods. 
Analytical and computational evidence in \cite{AD} suggests that the discretized 
characteristic boundary conditions, although not exactly transparent, are 
nevertheless highly absorbent. We note that the same type of characterisic 
absorbing conditions may be used for the \eqref{eq:SW} over a variable bottom, 
at least in the case where the bottom is locally horizontal at the endpoints 
cf.\ e.g.\ \cite{KD} and its references. 

Finding (exact) transparent boundary conditions for the \eqref{eq:CB} is not 
easy, as a nonlocal problem should be solved for this nonlinear system. In 
practice, for small $\mu$, it is reasonable to assume that the Riemann 
invariants do not change much over short distances along the characteristics, 
and, consequently, to pose the b.c.\ \eqref{eq:3p1} as approximate, absorbing 
b.c.'s for \eqref{eq:CB} as well. This has been widely done in practice, for 
example in numerical simulations of the Serre equations cf.\ e.g.\ \cite{CBM2}, 
\cite{BCLMT}; in \cite{DM} the related problem of deriving one-way 
approximations of the Serre equations is discussed. Our aim in this subsection 
is to assess, by numerical experiment, the accuracy of \eqref{eq:3p1} as 
approximate absorbing boundary conditions for the \eqref{eq:CB}, paying special 
attentions to their efficacy in simulating outgoing {\em solitary-wave} 
solutions of the \eqref{eq:CB}. 

In order to derive (classical) solitary-wave solutions of \eqref{eq:CB} on the 
real line, we let $\zeta=\zeta_s(x-c_st)$, $u=u_s(x-c_st)$, where $c_s$ is the 
speed of the solitary wave and $\zeta_s(\xi)$, $u_s(\xi)$ are smooth functions 
that tend to zero, along with their derivatives, as $|\xi|\to\infty$. Inserting 
these expressions in \eqref{eq:CB} and integrating we see that the equations for 
$\eta_s$ and $u_s$ decouple and give 
\begin{equation}\label{eq:3p2} 
\zeta_s=\frac{u_s}{c_s-\varepsilon u_s},\quad \frac{c_s\mu}{3}u''_s + \frac 
\varepsilon 2 u_s^2 - c_su_s + \frac{u_s}{c_s-\varepsilon u_s} = 0, 
\end{equation} 
A further integration yields that $u_s$ satisfies the ode 
\begin{equation}\label{eq:3p3} 
\frac{c_s\mu}{6}(u_s')^2 + \frac \varepsilon 6 u_s^3 - \frac{c_s}{2}u_s^2 - 
\frac 1 \varepsilon u_s - \frac{c_s}{\varepsilon^2}\ln \frac{c_s-\varepsilon 
u_s}{c_s} = 0. 
\end{equation} 
It is straightforward to see that $\zeta_s$ and $u_s$ have a single positive 
maximum at some point $\xi_0$ (we assume that $\xi_0=0$). Denoting 
$A=\max\zeta_s$, $B=\max u_s$, we get 
\begin{equation}\label{eq:3p4} 
A=\frac{B}{c_s-\varepsilon B},\quad \varepsilon B^3-\frac{c_s}{2}B^2 -\frac 1 
\varepsilon B - \frac{c_s}{\varepsilon^2}\ln\left(\frac{c_s-\varepsilon 
B}{c_s}\right) = 0, 
\end{equation} 
from which one may compute the speed-amplitude relation 
\begin{equation}\label{eq:3p5} 
c_s = \frac{\sqrt 6(1+\varepsilon A)}{\sqrt{3+2\varepsilon A}} 
\frac{\sqrt{(1+\varepsilon A)\ln(1+\varepsilon A) - \varepsilon A}}{\varepsilon 
A}. 
\end{equation} 
For fixed $\varepsilon$, $c_s$ is monotonically increasing with $A_s$ but stays 
below the  straight line $c_s=1+\frac{\varepsilon A}{2}$, which is the 
speed-amplitude relation of the solitary waves of the Serre equations. (The 
formulas \eqref{eq:3p2}--\eqref{eq:3p5} were derived in \cite{AD1} in the case 
of the unscaled \eqref{eq:CB}. Note that there are some typographical errors in 
\cite{AD1}: In equation (1.58) of \cite{AD1} the last term in the left-hand 
side of the equation should have the sign $+$, while in the equation preceding 
\eqref{eq:3p2} in \cite{AD1} the third term in the left-hand side should have 
the sign $+$ and the last term the sign $-$. However formulae (3.1) and (3.2) 
of \cite{AD1}, which are the analogous of \eqref{eq:3p5} and \eqref{eq:3p4} 
above, are correct.) 

When $\varepsilon A_s$ is not large, i.e.\ when \eqref{eq:CB} is a valid model 
for surface waves, it may be seen by \eqref{eq:3p5} and also by numerical 
simulations that the solitary-wave solutions of \eqref{eq:CB} satisfy the 
subcriticality condition. (Since there is no closed-form formula for the 
solitary waves we generate them numerically by solving for given $c_s$ the 
nonlinear o.d.e.\ \eqref{eq:3p2} that $u_s$ satisfies, taking zero boundary 
conditions for $u_s$ and $u_s'$ at the endpoints of a large enough spatial 
interval using the routine {\tt bvp4c} of \cite{MATLAB:2018}.) 

In the numerical experiments to be described in the sequel we solved the 
\eqref{eq:SW} and the \eqref{eq:CB}, unless otherwise specified, on the spatial 
interval $[0,50]$ using cubic splines on a uniform mesh with $h=0.025$, coupled 
with RK4 time stepping with time step satisfying $\frac k h = \frac 1 2$, up to 
$T=50$. 

We set the stage by solving numerically the \eqref{eq:SW} with $\varepsilon=1$ 
with the b.c.\ \eqref{eq:3p1}, posed now at the endpoints of $x=0$ and $x=50$. 
As initial condition we take the solitary wave of \eqref{eq:CB} with 
$\mu=\varepsilon=1$ of speed $c_s=1.18112$, centered at $x=25$, which we 
multiply by a factor $0.1$ (thus it is no longer a solitary wave), so that no 
discontinuities develop in its evolution under \eqref{eq:SW} for the duration of 
the experiment. As expected, the initial single-hump wave is split in two 
pulses: a larger one of amplitude of about $0.04$ traveling to the right with a 
speed of about $1.057$ and which starts exiting the computational interval at 
$x=50$ at about $t=22.5$, (the exit is completed by about $t=30$), and a smaller 
one of amplitude of about $0.0035$ that travels to the left with speed $1.005$ 
and exits the interval at $x=0$ at about $t=24.5$. 

In Figure \ref{fig:2} we present some graphs that are relevant for assessing the 
accuracy of the absorbing b.c.'s for this example. (All graphs refer to 
$\zeta$.) In Fig.~\ref{fig:2a} we observe the temporal variation of the 
wavefield at $x=40$. 
\begin{figure}[htbp] 
\centering 
\begin{subfigure}{.49\textwidth} 
\centering 
\includegraphics[width=\textwidth]{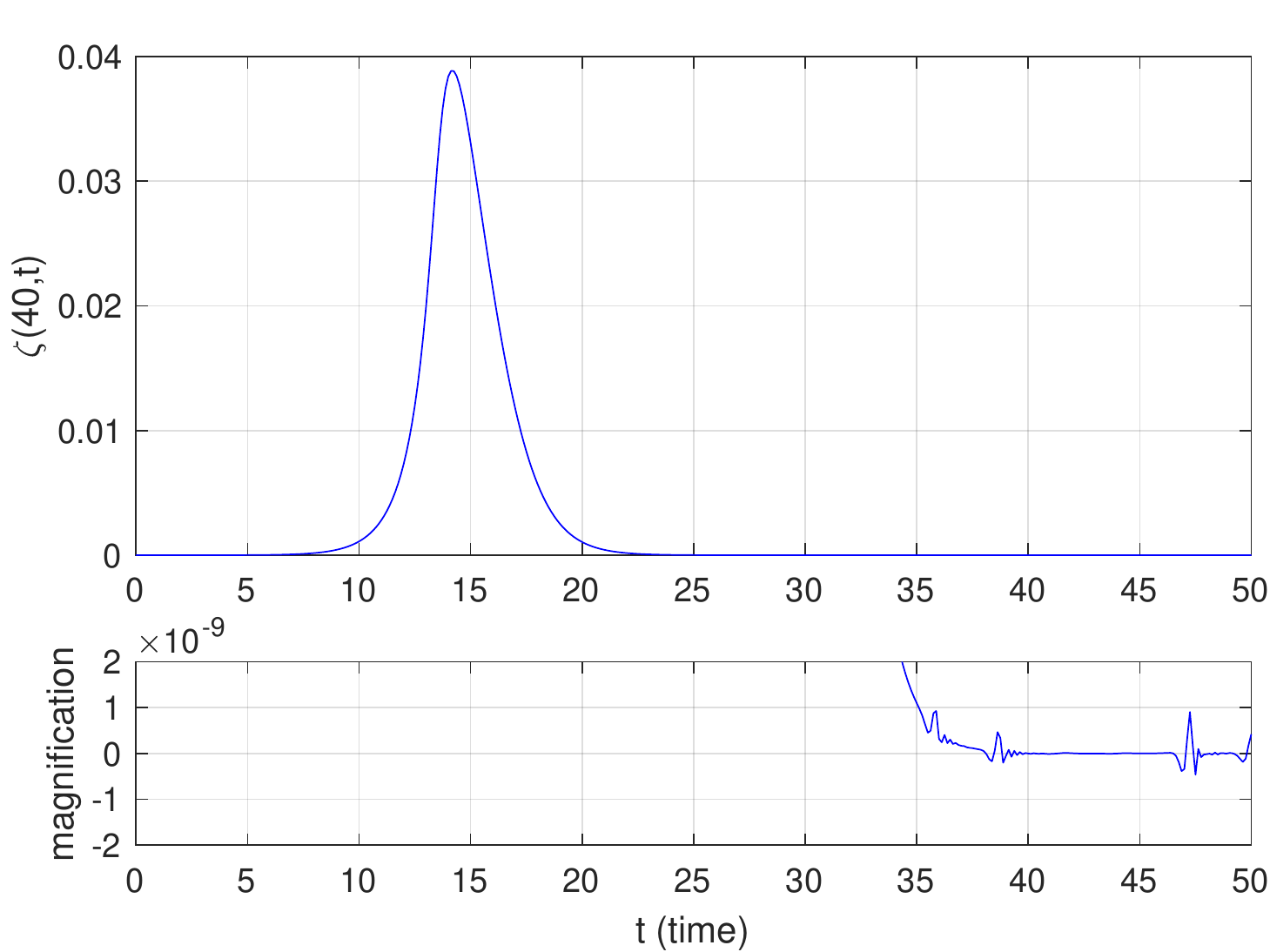}
\subcaption{\label{fig:2a}}
\end{subfigure}
\begin{subfigure}{.49\textwidth} 
\centering 
\includegraphics[width=\textwidth]{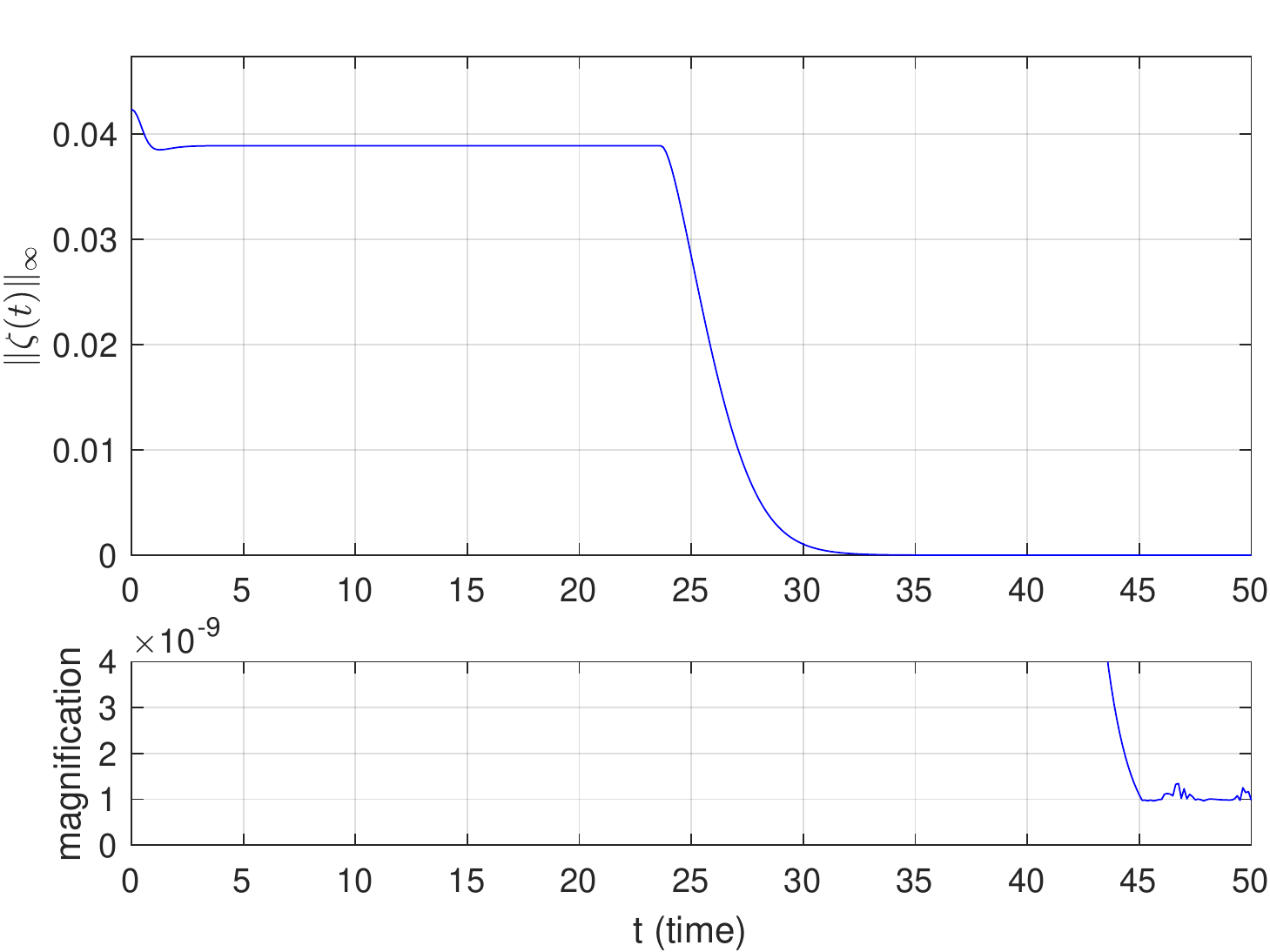}
\subcaption{\label{fig:2b}}
\end{subfigure}\\ 
\begin{subfigure}{.49\textwidth} 
\centering 
\includegraphics[width=\textwidth]{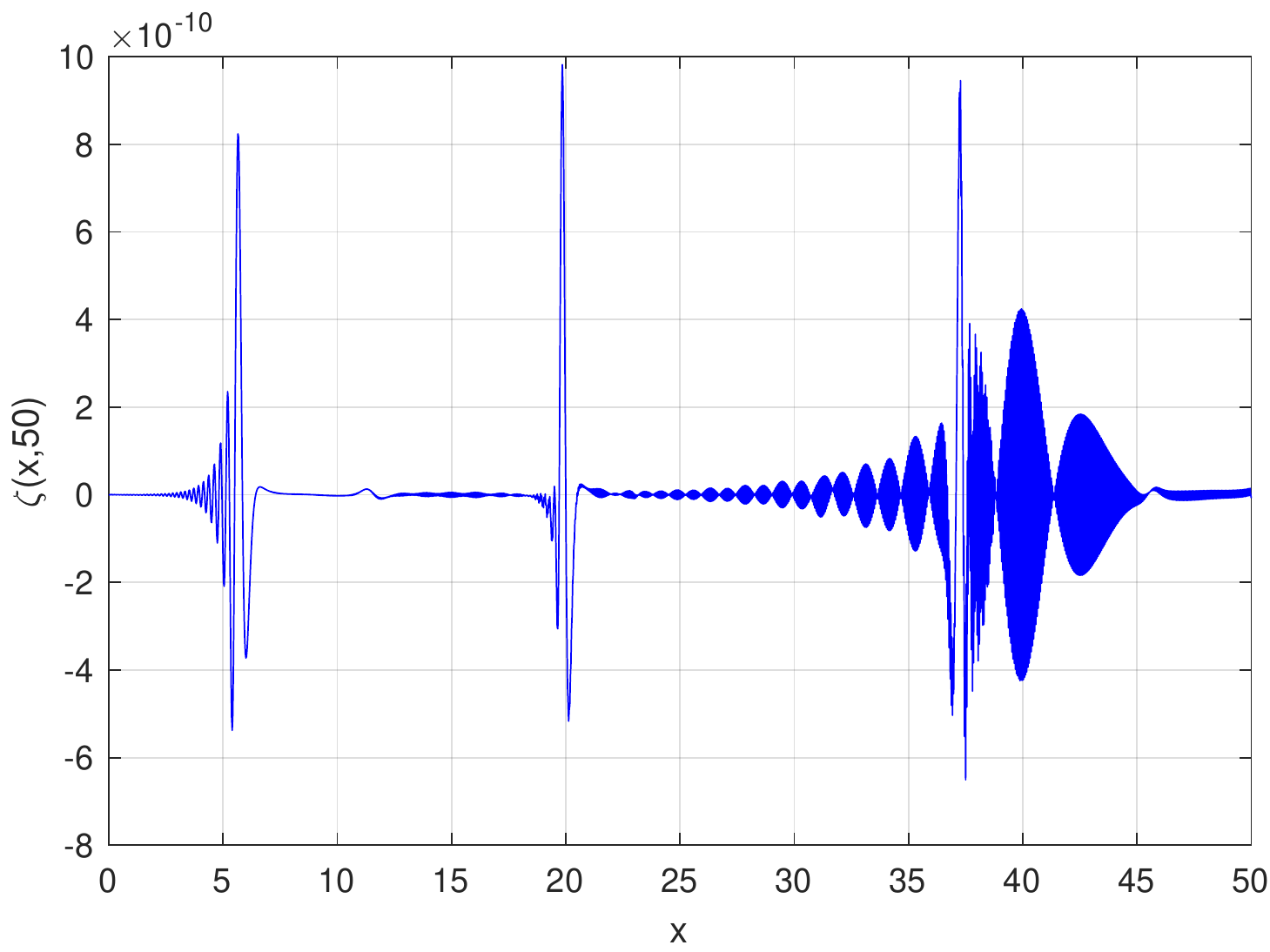}
\subcaption{\label{fig:2c}}
\end{subfigure}
\caption{Accuracy of the numerical characteristic b.c.'s for the \eqref{eq:SW}, 
$\varepsilon=1$, \subref{fig:2a}: $\zeta(40,t)$ with magnification underneath, 
\subref{fig:2b}: $\max_x\zeta(x,t)$ with magnification underneath, 
\subref{fig:2c}: Magnification of $\zeta(x,50)$ \label{fig:2}} 
\end{figure} 
The pulse that travels to the right passes this gauge and exits the interval. 
What remains after $t\simeq 30$ is a small residual consisting of 
small-amplitude oscillations reflected from the boundary due to the inexactness 
of the discretized b.c.'s and shown in the magnification of \ref{fig:2a} to be 
of $\mathcal O(10^{-9})$. In \ref{fig:2b} we show the maximum amplitude of 
$\zeta$ with respect to $x$ over the whole interval as a function of $t$, while 
\ref{fig:2c} shows the small oscillations still present in the computational 
interval at the end of the experiment ($t=50$). The are all of magnitude at most 
$10^{-9}$ and consist of a main wavepacket of high frequency and amplitude of 
about $4\times 10^{-10}$ centered at about $x=40$ and moving to the right, and 
three larger amplitude `thin' wavetrains of small support centered at about 
$x=5$ (moving to the right), $x=20$ (moving to the left) and $x=37.5$ (moving to 
the left), respectively. The main oscillatory wavepacket is produced when the 
right-traveling  pulse exits the boundary at $x=50$. This wavepacket moves to 
the left with speed equal to about $7$ and has undergone three reflections at 
the boundary by $T=50$. The thinner wavetrains (of speed about $1$) are 
generated by the interaction of this wavepacket with the boundaries (The 
left-traveling pulse produced by the splitting of the initial condition 
produces, when it hits the boundary at $x=0$, artificial reflections with 
amplitude well below $10^{-10}$.) 

In Figure \ref{fig:3}, resp.\ \ref{fig:4}, we show analogous graphs in the case 
of the \eqref{eq:CB} system in the cases $\varepsilon=\mu=0.1$, resp.\ 
$\varepsilon=\mu=0.01$. As initial condition we took now the exact solitary-wave 
profile of \eqref{eq:CB} for these values of $\varepsilon$, $\mu$, and of speed 
$c_s=1.18112$. As a consequence, the wave moves to the right without changing 
its shape. The fact that the characteristic b.c.'s are no longer exactly 
transparent for the continuous system is manifested by the larger magnitudes of 
the residual artificial oscillations, which are now of $\mathcal O(10^{-3})$, 
resp.\ $\mathcal O(10^{-4})$. (Note their dispersive character in the larger 
$\mu$ case, Fig.~\ref{fig:2c}.)  
\begin{figure}[htbp] 
\centering 
\begin{subfigure}{.49\textwidth} 
\centering 
\includegraphics[width=\textwidth]{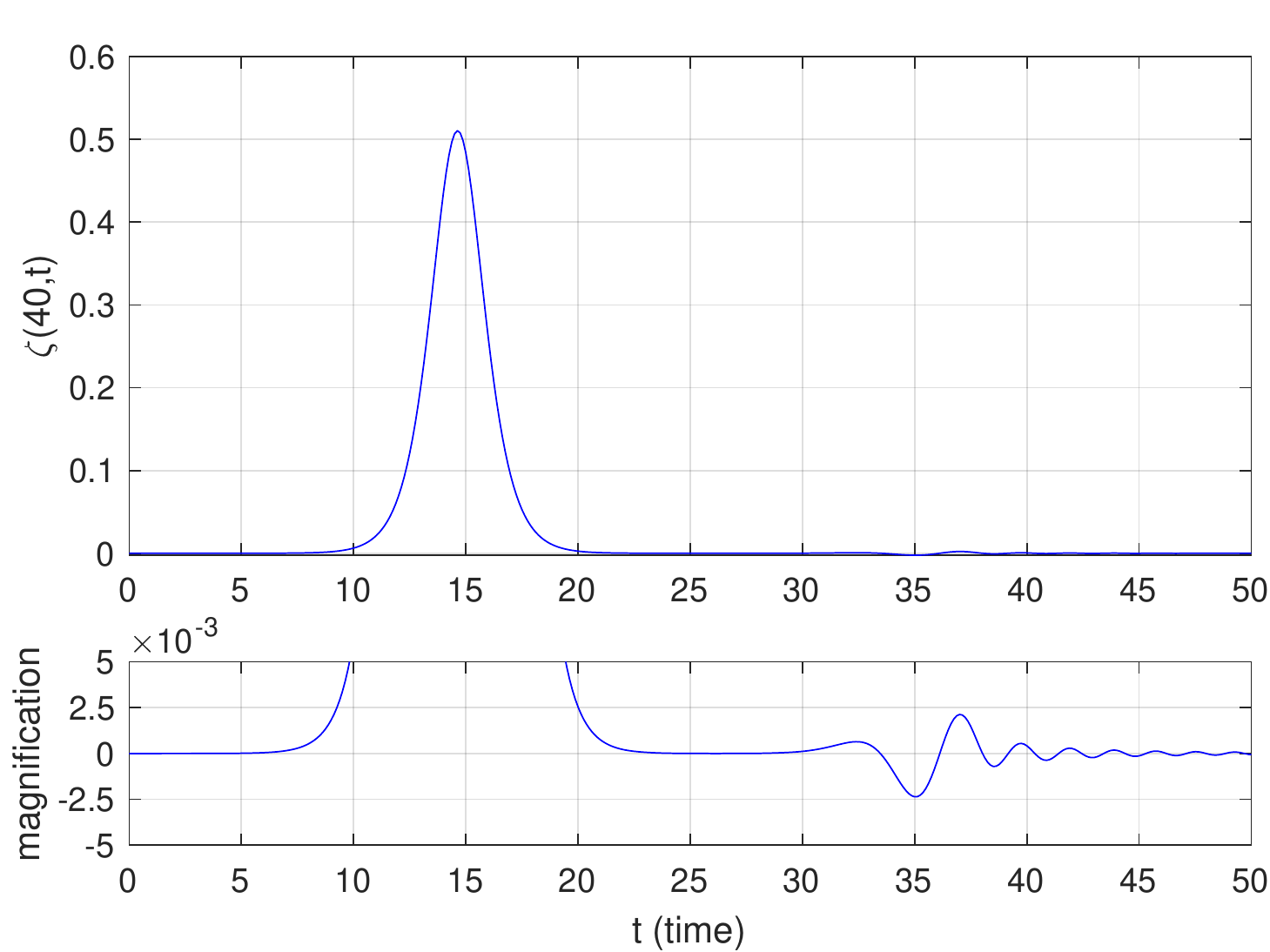}
\subcaption{\label{fig:3a}}
\end{subfigure}
\begin{subfigure}{.49\textwidth} 
\centering 
\includegraphics[width=\textwidth]{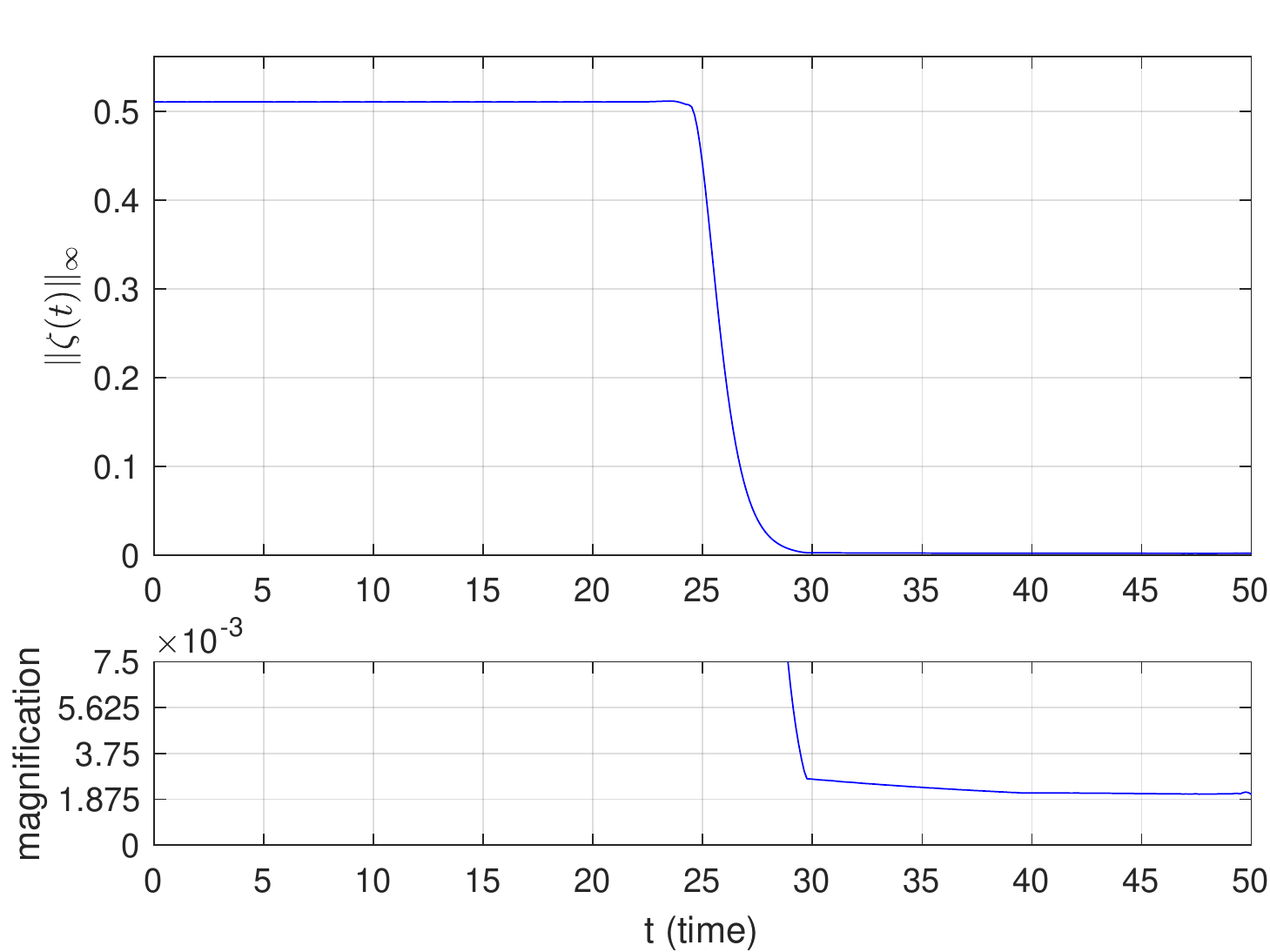}
\subcaption{\label{fig:3b}}
\end{subfigure}\\ 
\begin{subfigure}{.49\textwidth} 
\centering 
\includegraphics[width=\textwidth]{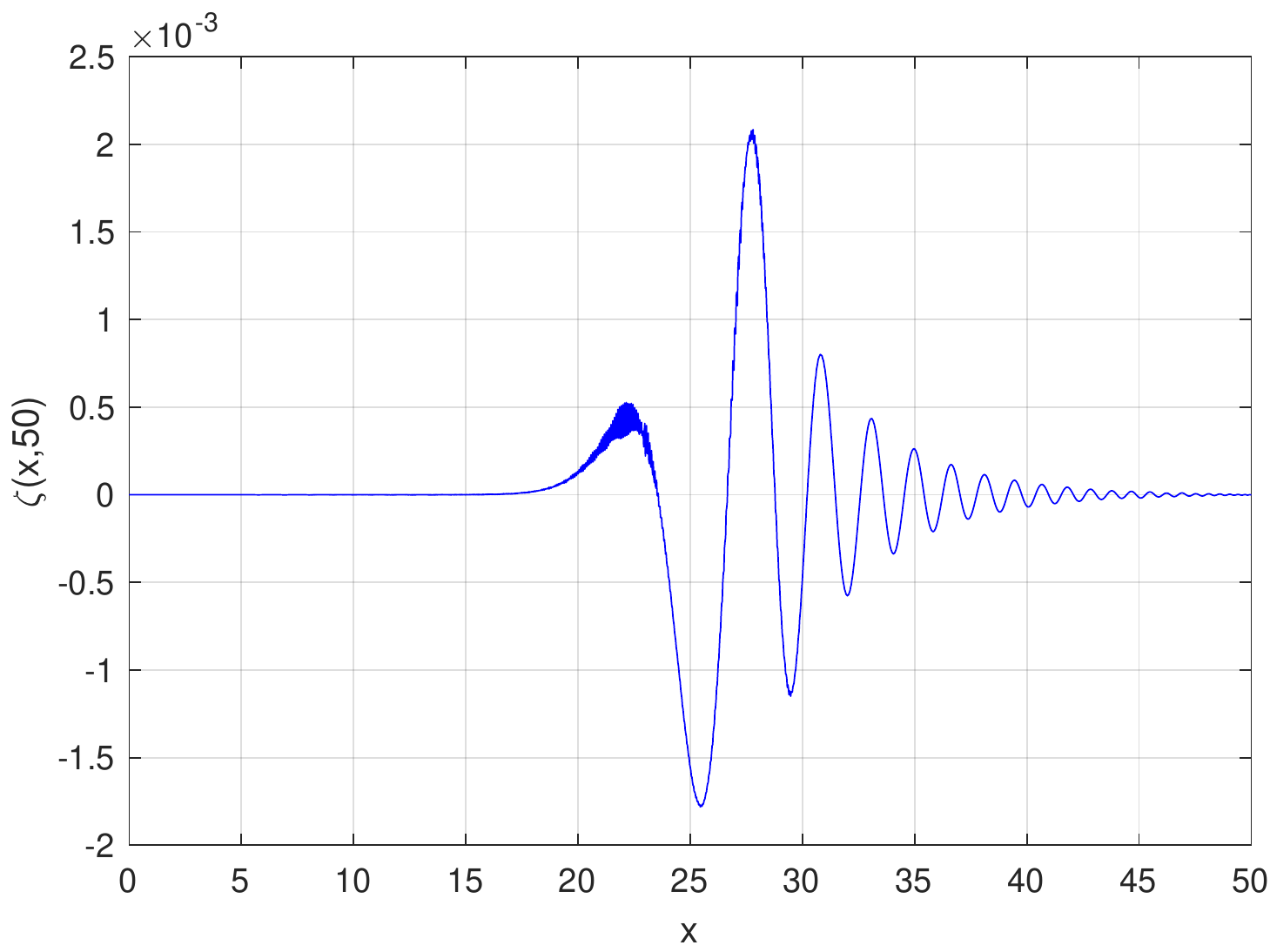}
\subcaption{\label{fig:3c}}
\end{subfigure}
\caption{Accuracy of the numerical characteristic b.c.'s for the \eqref{eq:CB}, 
$\varepsilon=\mu=0.1$, \subref{fig:3a}: $\zeta(40,t)$ with magnification 
underneath, \subref{fig:3b}: $\max_x\zeta(x,t)$ with magnification underneath, 
\subref{fig:3c}: Magnification of $\zeta(x,50)$ \label{fig:3}} 
\end{figure}
\begin{figure}[htbp] 
\centering 
\begin{subfigure}{.49\textwidth} 
\centering 
\includegraphics[width=\textwidth]{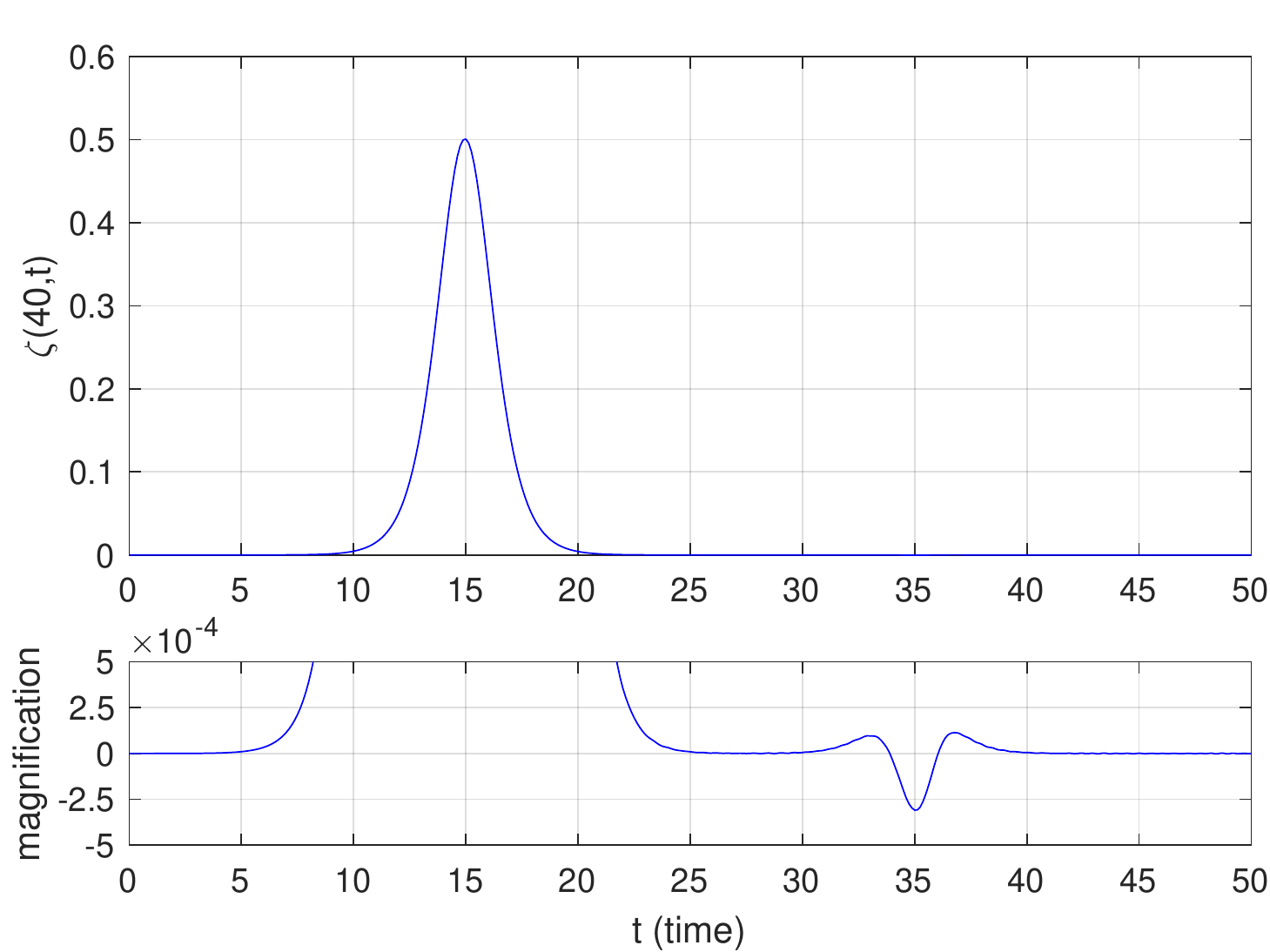} 
\subcaption{\label{fig:4a}}
\end{subfigure}
\begin{subfigure}{.49\textwidth} 
\centering 
\includegraphics[width=\textwidth]{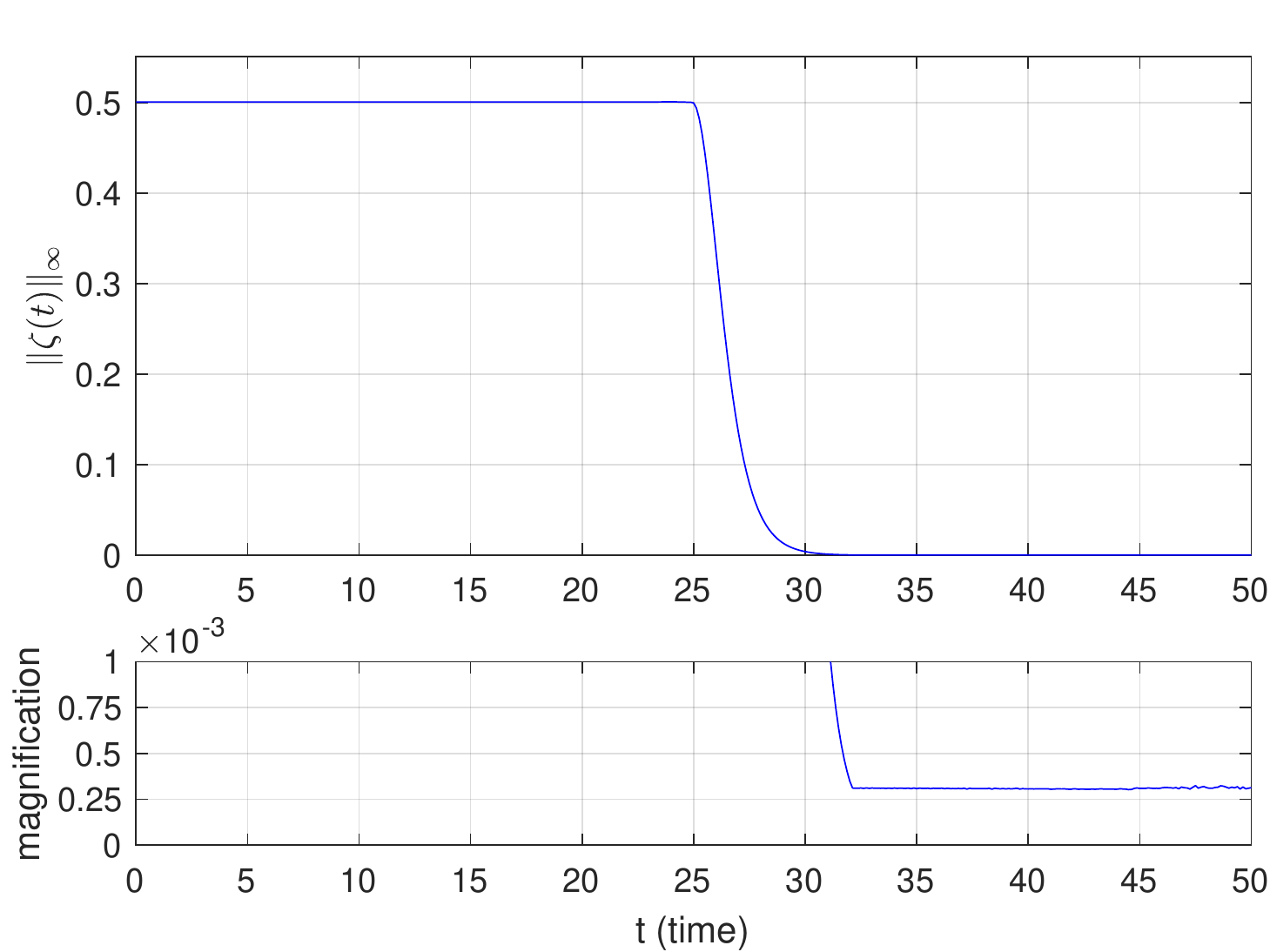} 
\subcaption{\label{fig:4b}}
\end{subfigure}\\ 
\begin{subfigure}{.49\textwidth} 
\centering 
\includegraphics[width=\textwidth]{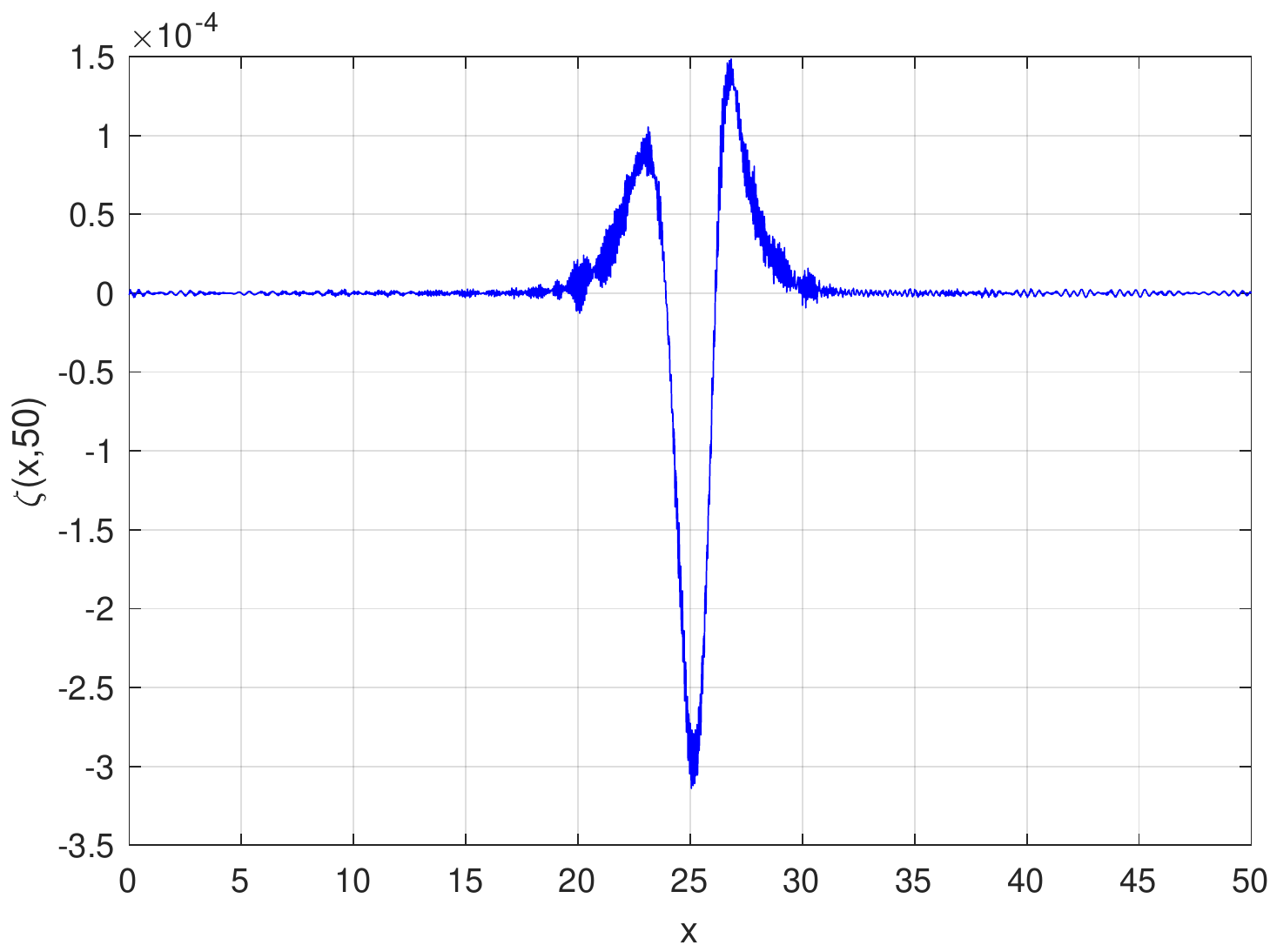} 
\subcaption{\label{fig:4c}}
\end{subfigure}
\caption{Accuracy of the numerical characteristic b.c.'s for the \eqref{eq:CB}, 
$\varepsilon=\mu=0.01$, \subref{fig:4a}: $\zeta(40,t)$ with magnification 
underneath, \subref{fig:4b}: $\max_x\zeta(x,t)$ with magnification underneath, 
\subref{fig:4c}: Magnification of $\zeta(x,50)$ \label{fig:4}} 
\end{figure}
The main pulse in graph (c) of Figures \ref{fig:3} and \ref{fig:4} is due to the 
modelling, i.e.\ the approximate character of the characteristic b.c.'s, while 
the superimposed noise in Fig.\ \ref{fig:4c} disappears as $h$ is decreased. The 
amplitude of the residual was equal to about $2.1\times 10^{-3}$ for 
$\varepsilon=\mu=0.1$ and fell to $3.2\times 10^{-4}$ for 
$\varepsilon=\mu=0.01$, and to $3.3\times 10^{-5}$ for $\varepsilon=\mu=0.001$ 
(figure not shown). We thus observe that it decreases linearly with $\mu$ when 
$\varepsilon=\mu$. As expected, for fixed $\varepsilon$ we observed that this 
amplitude decreased with $\mu$. For example, for $\varepsilon=0.01$ and 
$\mu=10^{-3}$ it was equal to about $3.6\times 10^{-5}$, for $\mu=10^{-4}$ it 
was of $\mathcal O(10^{-6})$. 

Our conclusion is that for small $\varepsilon=\mu$, i.e.\ when the \eqref{eq:CB} 
is a valid model, the (approximate) characteristic b.c.'s for the \eqref{eq:CB} 
are satisfactorily absorbing. We extended these b.c.'s in the case of the 
variable bottom models \eqref{eq:CBw} and \eqref{eq:CBs} and used them in 
numerical experiments with these systems that will be reported in the next 
subsection.

\subsection{Propagation of solitary waves over a variable bottom} 
\label{subsec:3p3} 

In this subsection we present the results of several numerical experiments we 
performed with the variable-bottom models \eqref{eq:CBw} and \eqref{eq:CBs} in 
order to validate the numerical methods used for their solution, compare the two 
models, and compare the results of \eqref{eq:CBs} with those obtained by the 
Serre-Green-Naghdi system and with experimental measurements. We mainly use test 
problems already considered in the literature, whose main theme is the study of 
the changes that solitary-wave pulses undergo when propagating over an uneven 
bottom.

\subsubsection{Solitary waves on a sloping beach} 
\label{subsubsec:3p3p1} 

We first consider the problem of a solitary wave climbing a sloping beach of 
mild slope that was studied by Peregrine in his pioneering study \cite{P}, in 
which he derived the \eqref{eq:CBs} system and solved it numerically by a finite 
differene scheme. In our experiments we used the \eqref{eq:CBs} in unscaled, 
nondimensional variables (i.e.\ setting $\varepsilon=\mu=1$) and solved it with 
our fully discrete scheme using cubic splines on a uniform mesh with $N=2000$ 
spatial intervals and $M=2N$ temporal steps. Following \cite{P} we consider, 
using out notation, a bottom of uniform slope $\alpha>0$ given by $\eta_b(x) = 
\alpha x$ on a spatial interval of the form $[0,L_\alpha]$. As initial condition 
we take as in \cite{P} a solitary wave of the form 
\begin{equation}\label{eq:3p6} 
\zeta_0(x) = a_0\operatorname{sech}^2\left[\tfrac 1 2 \sqrt{3a_0}(x-x_0)\right], 
\end{equation} 
where $x_0=1/\alpha$. This is a solitary wave of the KdV type equation 
$\zeta_t+\zeta_x + \frac 3 2 \zeta\zeta_x + \frac 1 6 \zeta_{xxx}=0$ with speed 
$c_s=1+a_0/2$. The KdV equation in this form is obtained as a one-way 
approximation of the \eqref{eq:CB} with $\varepsilon=\mu=1$ in the standard 
manner, cf.\ \cite{W}. The particular solitary wave \eqref{eq:3p6} is centered 
at $x_0=1/\alpha$, where the (undisturbed) water depth is equal to one. The 
initial velocity of the pulse, found by inserting \eqref{eq:3p6} in the 
continuity equation, is given by 
\begin{equation}\label{eq:3p7} 
u_0(x) = \frac{-\left(1+\frac 1 2 a_0\right)\zeta_0(x)}{\alpha x+\zeta_0(x)}. 
\end{equation} 
Thus the initial condition \eqref{eq:3p6}--\eqref{eq:3p7} is not an exact 
solitary-wave solution of \eqref{eq:CB} but a close approximation thereof. We 
took an interval of length $L_\alpha = 1/\alpha+20$ to ensure that the support 
of the initial pulse was well within the spatial interval of integration. At 
$x=0$ we used the b.c.\ $u=0$ (which produced no reflections as the wave did 
not reach the left boundary within the temporal range of the experiment), posed 
absorbing (characteristic) boundary conditions at $x=L_\alpha$, and ran the 
experiment up to $t=25$. 

During this temporal interval the wave moves to the left, steepens (wave 
`shoaling') and grows in amplitude; its evolution resembles that of Fig.\ 1 of 
\cite{P}, which corresponds to $\alpha=1/30$, $a_0=0.1$. We compared our 
numerical results with those of the finite-difference scheme of Peregrine (given 
in the Appendix of \cite{P}) that we implemented. (Note that there is a misprint 
in the last equation of this scheme in \cite{P}: In the discretization of the 
term $\eta_bu_x$, the denominator should be $4\Delta x$.) We observed that the 
maximum discrepancy in the amplitude of $\zeta$ approximated by the two methods 
occurred at $t=25$ where the values were $0.14100$ for our scheme and $0.13634$ 
for the scheme of \cite{P} (implemented with $\Delta x=\Delta t=0.1$), which 
corresponds to a difference of about $3.4\%$ (Fig.\ 1 of \cite{P} shows a 
$\zeta$-amplitude of about $0.15$ at $t=25$ which does not correspond to the 
actual numerical results that the scheme of \cite{P} gives and is probably due 
to some inaccuracy in the graphics.) 

We also repeated with our scheme the numerical experiments leading to Fig.\ 2 of 
\cite{P} that depicts the change of amplitude of the solitary wave with depth 
for various values of $a_0$ in the case of a beach of slope $\alpha=1/20$. There 
was good agreement for low values of $a_0$; however the values given in \cite{P} 
for $a_0=0.2$ seem too high as the depth approaches $0.4$. (All the amplitudes 
computed by our scheme stay below the curve of Green's law for depths larger 
than $0.5$.) 

As the solitary wave climbs the sloping beach a small-amplitude flat wave of 
elevation is reflected backwards due to the presence of the sloping bed. The 
results of our computations agree with those of Fig.\ 3 in \cite{P}. Peregrine, 
op.\ cit., derives an approximate expression for the amplitude of the reflected 
wave of the form 
\begin{equation}\label{eq:3p8} 
\zeta_{\max,\text{refl}} \simeq \frac 1 2 \alpha \left(\frac  1 3 
a_0\right)^\frac 1 2, 
\end{equation} 
using characteristic variables for the linearized shallow water equations. We 
found quite a good agreement between our numerical results and the values 
computed by \eqref{eq:3p8}. For example, for $\alpha=1/40$, $a_0=0.1$, our 
computations gave $\zeta_{\max,\text{refl}}=0.0023$, while \eqref{eq:3p8} gives 
$0.0025$. We will return to the reflections due to the uneven bottom in 
subsection~\ref{subsubsec:3p3p3} in the sequel. 

As was previously mentioned, we used the approximate characteristic boundary 
conditions discussed in subsection \ref{subsec:3p2} at the right-hand boundary 
$x=L_\alpha$. We found that the b.c.\ also works for a sloping bottom provided 
the length of the domain is taken sufficiently large so that the artificial 
oscillations created at the boundary do not interfere as they travel to the left 
with the reflected wave due to the slope. As an example we consider a beach of 
slope $\alpha=1/40$ on the spatial interval $[0,70]$. As initial condition we 
took $\zeta(x,0)=\zeta_0(x)$ given by \eqref{eq:3p6} with $x_0=40$, $a_0=0.1$, 
and $u(x,0)=0$, i.e.\ a `heap' of water, so that sizeable pulses are generated 
and propagate in both directions. 
\begin{figure}[htbp] 
\centering 
\begin{tikzpicture}\small 
\begin{axis}[
width=.95\textwidth, height=.4*.95\textwidth, 
mark=none, line width=0.5pt, 
grid=major, 
enlargelimits=false, ymin=-1e-3, ymax=4e-3, ytick distance=1e-3, 
xticklabel={\pgfmathprintnumber[assume math mode=true]{\tick}}, 
yticklabel={\pgfmathprintnumber[assume math mode=true]{\tick}}, 
xlabel={$x$},ylabel={$\zeta$}
]

\addplot[color=blue] 
table[x={xn}, y={eta}] 
{figs/peregrine_fnl_abs_a_40_a0_p1_xN_70_T_50_u0_0_magnif.txt}; 

\addplot[color=red, style=dotted] 
coordinates {(20,5.095294e-04) (70,5.095294e-04)}; 

\addplot[color=black, style=dotted] 
coordinates {(20,0) (70,0)}; 

\draw[<-] (28,.6e-3) -- (29.75,1.85e-3) node[above] 
{\setlength{\fboxsep}{0pt}\colorbox{white}{\parbox{10em}{ Reflection from 
the\\[-1.5pt] left-traveling pulse due\\[-1.5pt] to the sloping bottom}}}; 

\draw[->] (51,1.3e-3) -- (44,1.3e-3); 
\node[above] at (46.5,1.85e-3) 
{\setlength{\fboxsep}{0pt}\colorbox{white}{\parbox{9em}{Artificial 
reflection\\[-1.5pt] from right-travelling\\[-1.5pt] pulse (travelling left)}}}; 

\draw[->] (61.5,1.1e-3) -- (67.5,1.1e-3); 
\node[above] at (62.1,1.4e-3) 
{\setlength{\fboxsep}{0pt}\colorbox{white}{\parbox{9em}{Dispersive tail 
of\\[-1.5pt] right-travelling pulse\\[-1.5pt] about to reach bndry}}}; 

\end{axis}
\end{tikzpicture}
\caption{Magnification of $\zeta$ reflections near the right boundary, $t=50$ 
\label{fig:5}}
\end{figure}
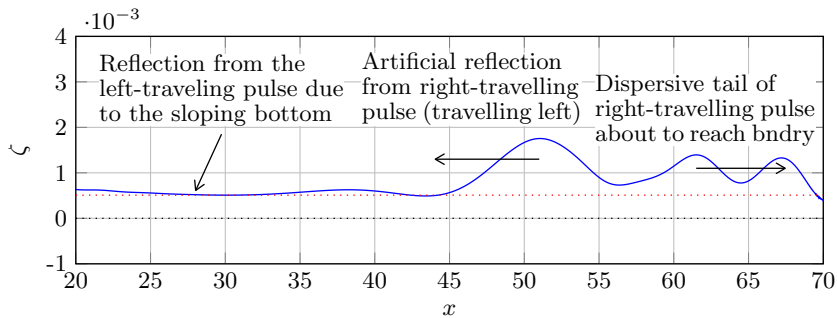 
Figure \ref{fig:5} shows a magnified profile of the surface elevation $\zeta$ as 
a function of $x$ in the interval $[20,70]$ at $t=50$, by which time the 
right-travelling pulse has left the domain. In the interval $[20,45]$ we observe 
the small-amplitude (of height approximately $5\times 10^{-4}$) reflection due 
to interaction of the left-travelling pulse with the sloping bottom. In the 
interval $[45,60]$ we observe the artificial oscillations reflected from the 
right-hand boundary at $x=70$ due to the approximate absorbing b.c.\ after the 
exit of the main right-travelling pulse. The ratio of the amplitude of the 
artificial reflection to that of the main pulse is about $4\%$. Finally, one may 
also observe on the extreme right the dispersive-tail oscillations that follow 
the main right-travelling pulse as they exit the domain.

\subsubsection{Transformation of a solitary wave propagating onto a shelf} 
\label{subsubsec:3p3p2} 

We next consider in detail an example of the transformation that a solitary wave 
undergoes as it propagates over a bottom of shelf type like the one shown in 
Figure \ref{fig:6}. This test problem was considered by Madsen and Mei in 
\cite{MM}. In this subsection we work in dimensionless, unscaled variables with 
$\varepsilon=\mu=1$. 

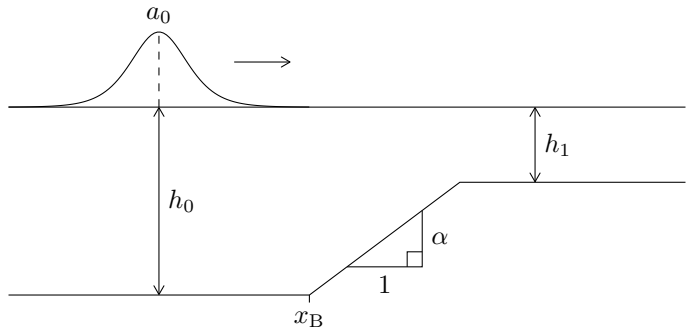
\begin{figure}[htbp] 
\centering 
\begin{tikzpicture} 
\draw (-2,0) -- (7,0); 

\draw (-2,-2.5) -- (2,-2.5) -- (4,-1) -- (7,-1); 
\draw (2,-2.5) -- (2,-2.6) node[below] {$x_\text{B}$}; 

\draw[<->] (0,-2.5) -- (0,0); 
\node[right] at (0,-1.25) {$h_0$}; 
\draw[dashed] (0,0) -- (0,1); 
\node[above] at (0,1) {$a_0$}; 

\draw[<->] (5,-1) -- (5,0); 
\node[right] at (5,-.5) {$h_1$}; 

\draw[domain=-2:2, samples=100] plot(\x,{1/cosh(\x*2)^2}); 
\draw[->] (1,.6) -- (1.75,.6); 

\draw (2+.5,-2.5+.75*.5) -- (2+1.5,-2.5+.75*.5) -- (2+1.5,-2.5+.75*1.5); 
\draw (2+1.5-.2,-2.5+.75*.5) -- (2+1.5-.2,-2.5+.75*.5+.2) -- 
(2+1.5,-2.5+.75*.5+.2); 
\node[below] at (3,-2.5+.75*.5) {$1$}; 
\node[right] at (3.5,-2.5+.75) {$\alpha$}; 
\end{tikzpicture} 
\caption{Solitary wave propagating onto a shelf \label{fig:6}} 
\end{figure}

The initial elevation of the solitary wave is given again by \eqref{eq:3p6}, in 
which $x_0$ is taken far enough from the toe of the sloping part of the bottom 
at $x=x_\text{B}$, so that $\zeta_0(x_\text{B})/a_0\ll 1$. The initial velocity 
is found again from the continuity equation but is now computed for a bottom of 
constant depth $h_0=1$, i.e.\ as 
\begin{equation}\label{eq:3p9} 
u_0(x) = \frac{\left(1+\tfrac 1 2 a_0\right)\zeta_0(x)}{1+\zeta_0(x)}. 
\end{equation} 
The solitary wave travels to the right, changes in amplitude and shape as it 
climbs the slope, and resolves itself into a sequence of solitary-wave pulses as 
it travels on the shelf of uniform depth $h_1<1$, cf.\ Figure~\ref{fig:9}. 

In \cite{MM} the pde model used was a Boussinesq system of KdV-BBM type with 
variable-bottom terms originally derived in \cite{MLeM}, and which, in the case 
of horizontal bottom, is locally well-posed, cf.\ \cite{BCS2}. The initial-value 
problem was integrated with a type of a method of `characteristics'. 
In order to form some idea of the proximity of the model used in \cite{MM} to 
\eqref{eq:CBs} we integrated both systems using our fully discrete scheme with 
cubic splines and RK4 time stepping over a variable bottom domain like that of 
Figure \ref{fig:6} with $0\leq x\leq 150$, $x_B=60$, $h_1=0.5$, $\alpha=1/20$. 
As initial values we took solitary waves of the respective systems of the same 
amplitude $a_0=0.12$ and centered at $x_0=30$. (Their speeds are very close but 
the wavelength of the solitary wave of the system of \cite{MM} was about $22\%$ 
larger. The difference of the two-solitary waves in $L^2$ was about $4.37\times 
10^{-2}$.) At the end of the computational domain at $t=22.5$, when both waves 
had climbed well onto the shelf and resolved themselves into two solitary waves 
plus dispersive tail, the two wavetrains had an $L^2$ distance of $5.53\times 
10^{-2}$, while the leading solitary waves had a difference in amplitude of 
about $3\times 10^{-3}$ and a phase difference (distance of positions of the 
crest) of $0.15$. We conclude that in the time scales of this and similar 
experiments typical solutions of the two systems stay close to each other, so 
that it is fair to compare in a general way the results of numerical experiments 
in \cite{MM} with similar ones that we ran with \eqref{eq:CBs} to be described 
in the sequel. 

We first make some quantitative remarks on the transformation of the solitary 
wave as it climbs on the sloping part of the bottom in Figure~\ref{fig:6}. As 
observed in subsection~\ref{subsubsec:3p3p1}, the amplitude of the solitary wave 
increases as the depth of the water decreases. In order to 
quantify this increase in the case of \eqref{eq:CBs} and our numerical method, 
and motivated by analogous experiments in \cite{MM}, we took $h_1=0.1$, 
$\alpha=1/20$, $0\leq x\leq 150$, $x_\text{B}=60$, and computed with cubic 
splines, $N=3000$, $M=2N$, the evolution (according to \eqref{eq:CBs}) of a 
solitary wave of \eqref{eq:CB} centered at $x=30$. We recorded the variation of 
the normalized amplitude $\zeta_{\max}/a_0$ of the solitary wave as a function 
of the water depth $\eta_b$ for various values of the initial amplitude $a_0$. 
In Figure~\ref{fig:7} we show the 
\begin{figure}[htbp] 
\centering 
\includegraphics[width=\textwidth]{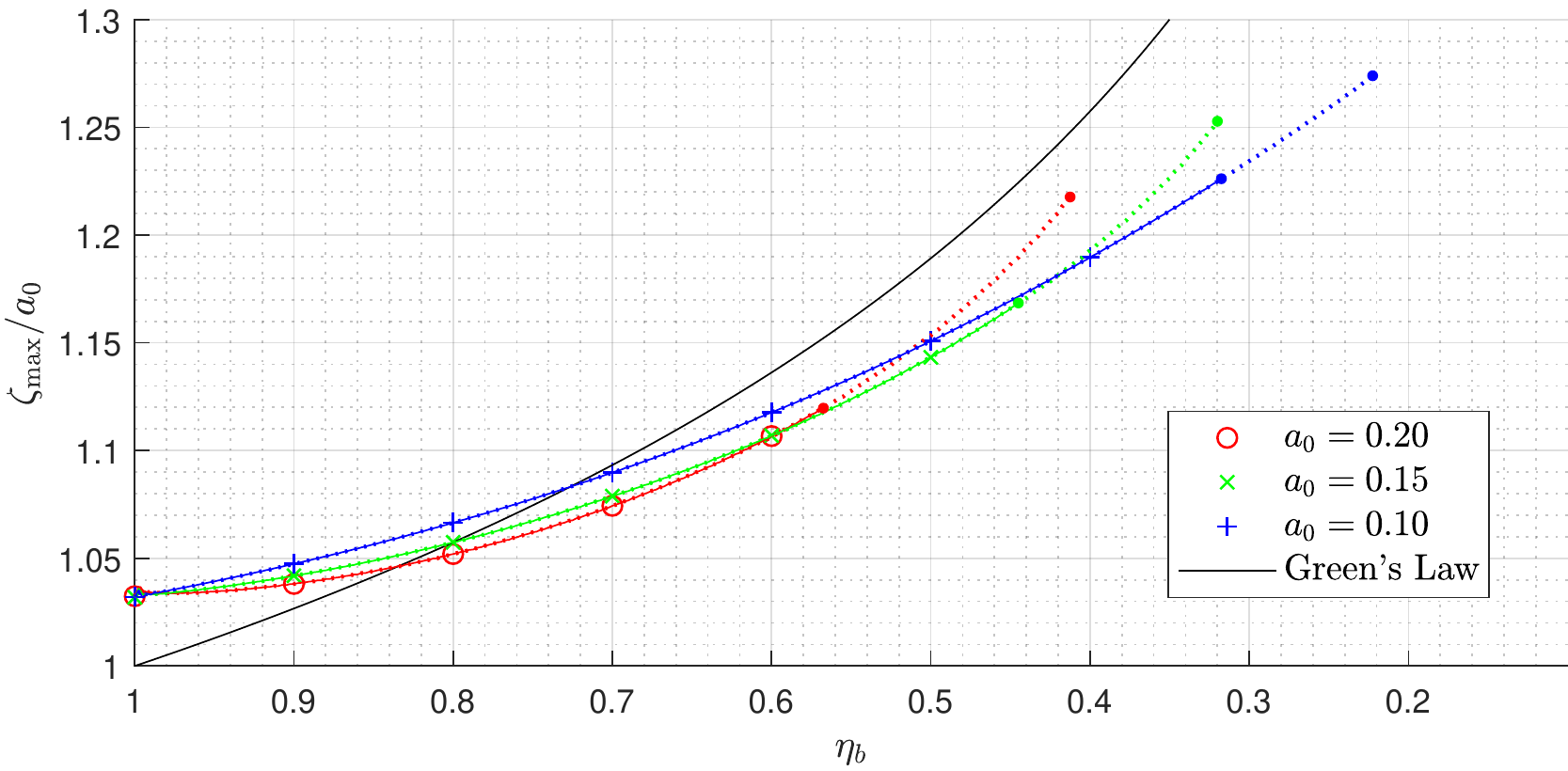}
\caption{Amplitude variation with depth for beach slope $\alpha=1/20$ for 
$a_0=0.2,\ 0.15,\ 0.1$. Computation stopping criteria: solid lines, 
$\max_x(\zeta(x,t)/\eta_b(x))<0.4$; dotted lines, 
$\max(\zeta(x,t)/\eta_b(x))<0.6$. \label{fig:7}} 
\end{figure} 
outcome of these numerical experiments corresponding to solitary waves of 
initial amplitudes $a_0=0.1$, $0.15$ and $0.2$. (The graph starts when the crest 
of the solitary wave is at $x=x_\text{B}$. At that point $\zeta_b=1$, but the 
forward point of the solitary wave is already travelling on the sloping bed; 
hence, the corresponding value of $\zeta_{\max}/a_0$ is about $1.034$ and not 1. 
For $\eta_b$ larger than about $0.6$ the three curves corresponding to the three 
amplitudes chosen are quite close to each other with the lowest initial 
amplitude $a_0=0.10$ giving the highest values of $\zeta_{\max}/a_0$. For 
$\eta_b$ smaller than about $0.5$ the sequence is reversed with the highest 
$a_0=0.2$ giving the highest $\zeta_{\max}/a_0$ values. The initial solid-line 
part of the three curves represents the values of $\zeta_{\max}/a_0$ up to the 
point where $\max_x\left(\tfrac{\zeta(x,t)}{\eta_b(x)}\right)=0.4$, which is 
probably a large upper bound of the range of validity of \eqref{eq:CBs}, while 
the dotted-line extensions of the curves go up to 
$\max_x\left(\frac{\zeta(x,t)}{\eta_b(x)}\right)=0.6$, which is probably beyond 
that range. We also show the curve of Green's law given by 
$\zeta_{\max}/a_0=\eta_b^{-1/4}$ for comparison purposes. It is to be noted that 
our results are in satisfactory agreement with those of the corresponding Fig.~3 
of \cite{MM} for values of $\eta_b$ in the range $1$ to $0.75$. 

These results are supplemented by those of Figure \ref{fig:8} in which we record 
the variation of $\zeta_{\max}/a_0$ as a function of $\eta_b$ for a solitary 
wave of fixed $a_0=0.1$ and slopes equal to $0.023$, $0.05$, and $0.065$. For 
$\eta_b$ larger than about $0.65$ 
\begin{figure}[htbp] 
\centering 
\includegraphics[width=\textwidth]{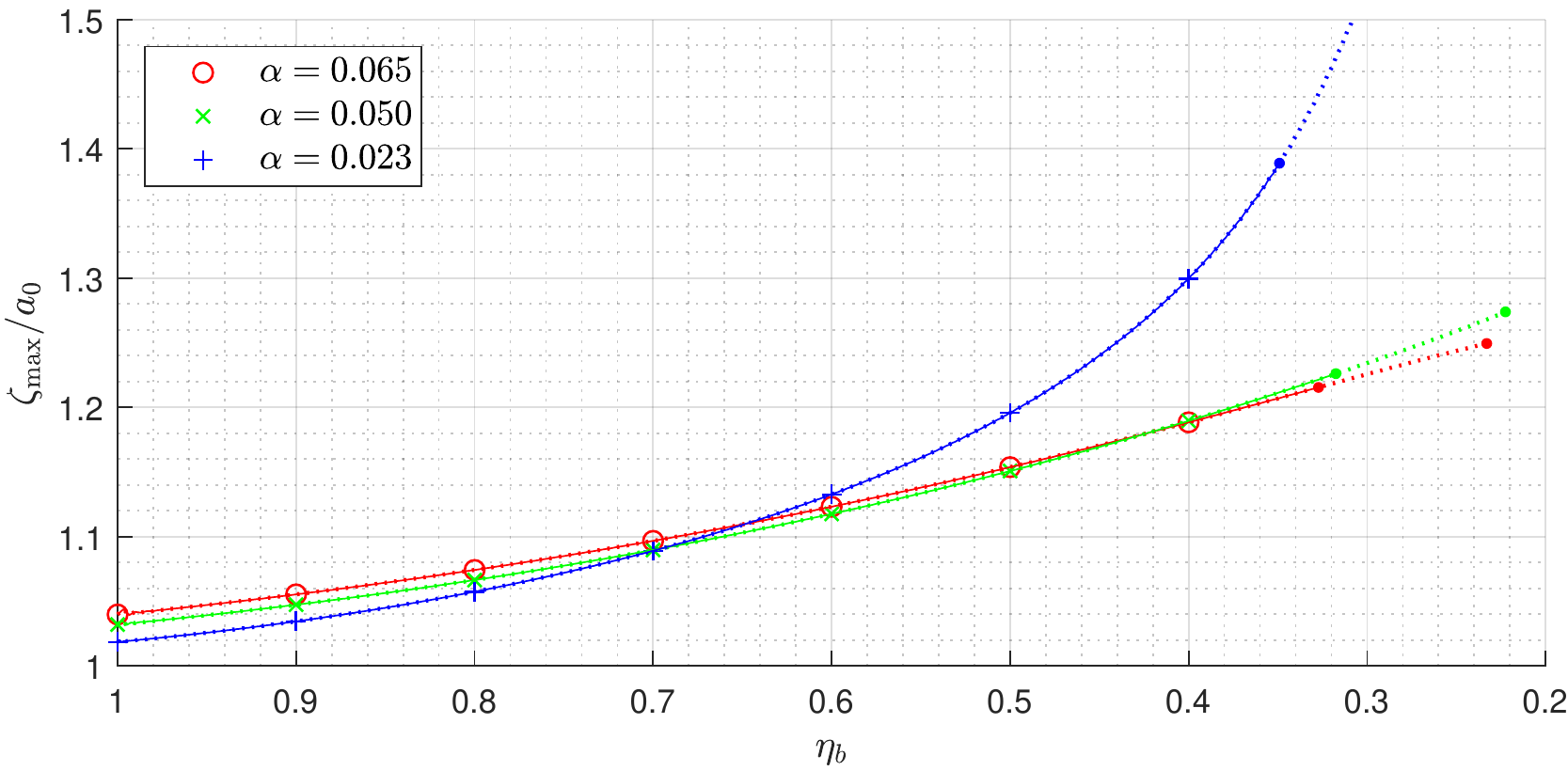}
\caption{Amplitude variation with depth for initial amplitude $a_0=0.1$ for 
slopes $\alpha=0.065,\ 0.05,\ 0.023$. Computation stopping criteria: solid 
lines, $\max_x(\zeta(x,t)/\eta_b(x))<0.4$; dotted lines, 
$\max_x(\zeta(x,t)/\eta_b(x))<0.6$. \label{fig:8}} 
\end{figure} 
all curves are fairly close to each other with the steeper slopes giving 
slightly higher values of $\zeta_{\max}/a_0$. For values of $\eta_b$ less than 
about $0.65$ the smaller slope gives the highest ratio $\zeta_{\max}/a_0$ while 
the two other curves remain close together (stopping criteria as in Figure 
\ref{fig:6}). A qualitatively similar behavior is observed in the analogous 
Figure 4 of \cite{MM}. 

The distortion the solitary wave suffers as it travels upslope causes the wave, 
when it reenters a horizontal-bottom region reaching the shelf, to resolve 
itself into a sequence of solitary waves followed by dispersive oscillations. 
This phenomenon was noticed in \cite{MM} for the model used in that paper, and 
is also present in our case of the \eqref{eq:CBs} system as well. In Figure 
\ref{fig:9} we show this phenomenon, which may be viewed as a manifestation of 
the stability 
\begin{figure}[htbp] 
\centering 
\includegraphics[width=\textwidth]{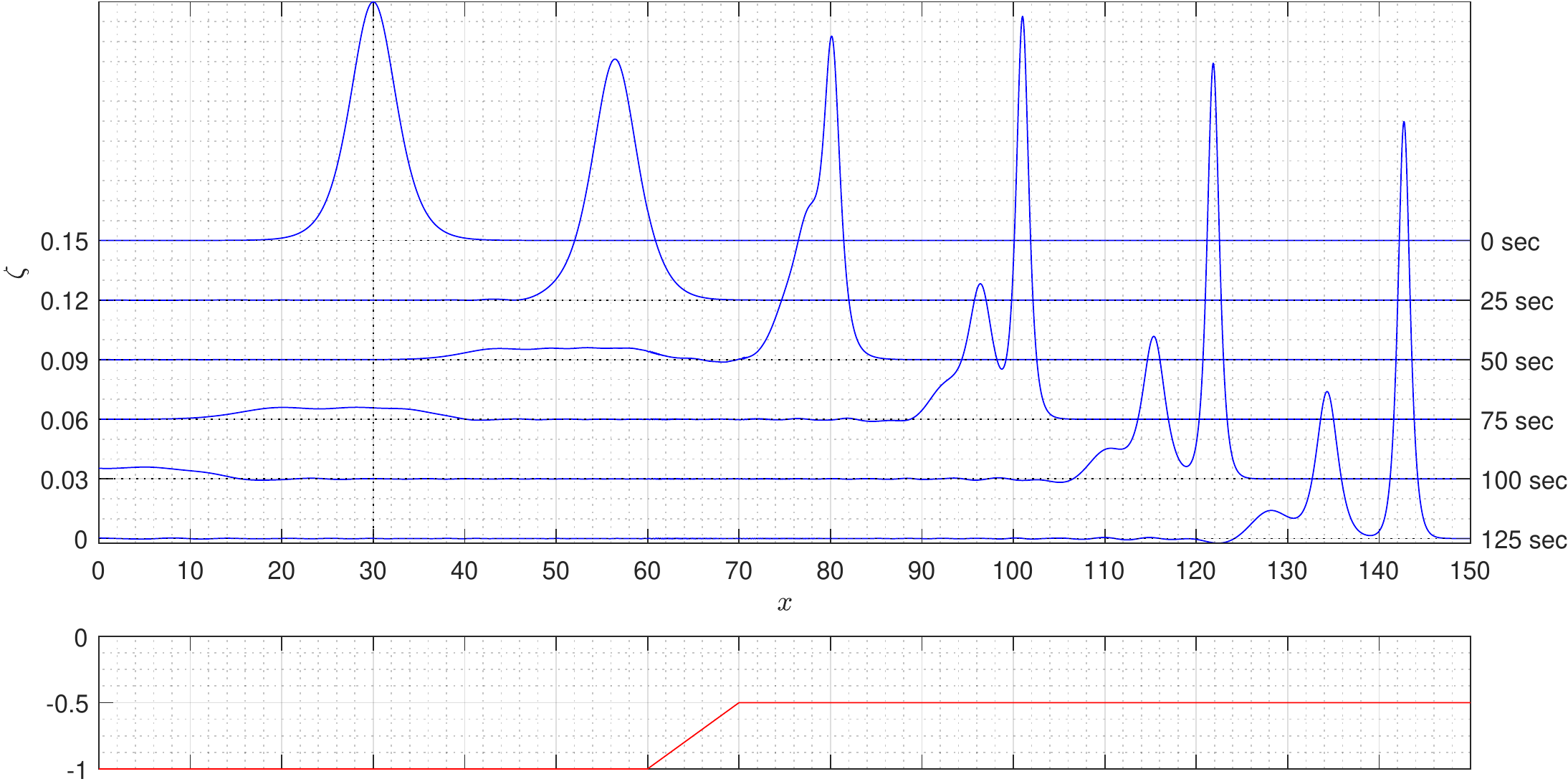} 
\caption{Transformation of a CB solitary wave ($a_0=0.12$) propagating up a slope 
of $\alpha=1/20$, onto a shelf of smaller depth, $h_1=0.5 h_0$. \label{fig:9}} 
\end{figure} 
of solitary waves of \eqref{eq:CB}. We took a spatial interval $[0,150]$, 
$h_1=0.5$, $x_B=60$, $\alpha=1/20$, and considered the evolution of a solitary 
wave of initial amplitude $a_0=0.12$. The graphs in Figure~\ref{fig:9} show the 
temporal evolution every $25$ temporal units (``seconds''). The solitary wave 
distorts as it climbs the sloping part of the bottom (depicted in the lower part 
of he graph), increases in amplitude, and by $t=125$ it has resolved itself into 
two solitary waves (a third is also possibly forming) plus a dispersive tail. 
The first solitary wave has an amplitude of about $0.2099$ and travels at a 
speed of about $0.84$. (We checked that it is indeed a CB-solitary wave.) This 
wavetrain is followed by the usual for upsloping environments flat reflection 
wave that travels to the left. The results of a similar experiment in \cite{MM} 
are qualitatively the same.

\subsubsection{Reflection and dispersion from various types of variable bottom} 
\label{subsubsec:3p3p3} 

As already mentioned in subsection~\ref{subsubsec:3p3p1}, when a solitary wave 
propagates up a sloping bottom, a small-amplitude, flat wave of elevation is 
generated by reflection from the uneven bottom and travels in the opposite 
direction. This phenomenon has been shown e.g.\ in Figs \ref{fig:5} and 
\ref{fig:9}. (In this subsection we work again in dimensionless, unscaled 
variables with $\varepsilon=\mu=1$.) Using characteristic variables theory for 
the linearized shallow water equations, in addition to the approximate formula 
\eqref{eq:3p8} for the reflected wave, Peregrine predicted in \cite{P} that the 
reflected wave will have a wavelength of about $2L$ if the slope occurs over a 
horizontal interval of length $L$. In order to check these results we integrated 
the \eqref{eq:CBs} over the variable bottom shown in the lower graph of 
Figure~\ref{fig:9} with an initial solitary wave of \eqref{eq:CB}, varying the 
slope and the initial amplitude $a_0$  of the wave; we present the results in 
Table~\ref{tab:3} that shows the amplitudes and wavelengths of the reflected 
wave predicted in \cite{P} and the numerical results given by our code. (Due to 
the shape of the reflected wave we measured its length by the formula 
$\frac{1}{|I|}\int_I\zeta\,\mathrm dx$, where $I=\{x:\zeta>0.8\,\zeta_{\max}\}$, 
at a short time after the full reflected wave had formed. In the case 
$\alpha=1/40$, $a_0=0.18$, we took $I=\{x:\zeta>0.6\,\zeta_{\max}\}$.) We 
conclude that 
\begin{table}[!htbp] 
\centering
\begin{tabular}{|*{6}{c|}} \hline
$\alpha$ & $a_0$ & $L$ & \parbox{4.5em}{refl.\ ampl.\\[-1pt] by \eqref{eq:3p8}} 
& \parbox{4em}{reflected\\[-1pt] amplitude} & \parbox{5em}{reflected\\[-1pt] 
wavelength} \\ \hline
1/20 & .12 & 10 & 5.000e-3 & 5.578e-3 & 22.35 \\ \hline 
1/40 & .12 & 20 & 2.500e-3 & 2.875e-3 & 43.00 \\ \hline
1/20 & .18 & 10 & 6.124e-3 & 6.880e-3 & 21.25 \\ \hline
1/40 & .18 & 20 & 3.062e-3 & 3.451e-3 & 41.65 \\ \hline
\end{tabular} 
\caption{Predicted and numerical values of amplitude and wavelength of reflected 
wave. \label{tab:3}}
\end{table}
the predictions of \cite{P} underestimate by a small amount the actual numerical 
results. 

In \cite{P} Peregrine also made some qualitative comments about the type of 
reflected waves generated by various kinds of uneven bottoms. We verified his 
general statements by performing various numerical 
\begin{figure}[hptb] 
\centering 
\begin{subfigure}{.9\textwidth} 
\centering 
\includegraphics[width=\textwidth]{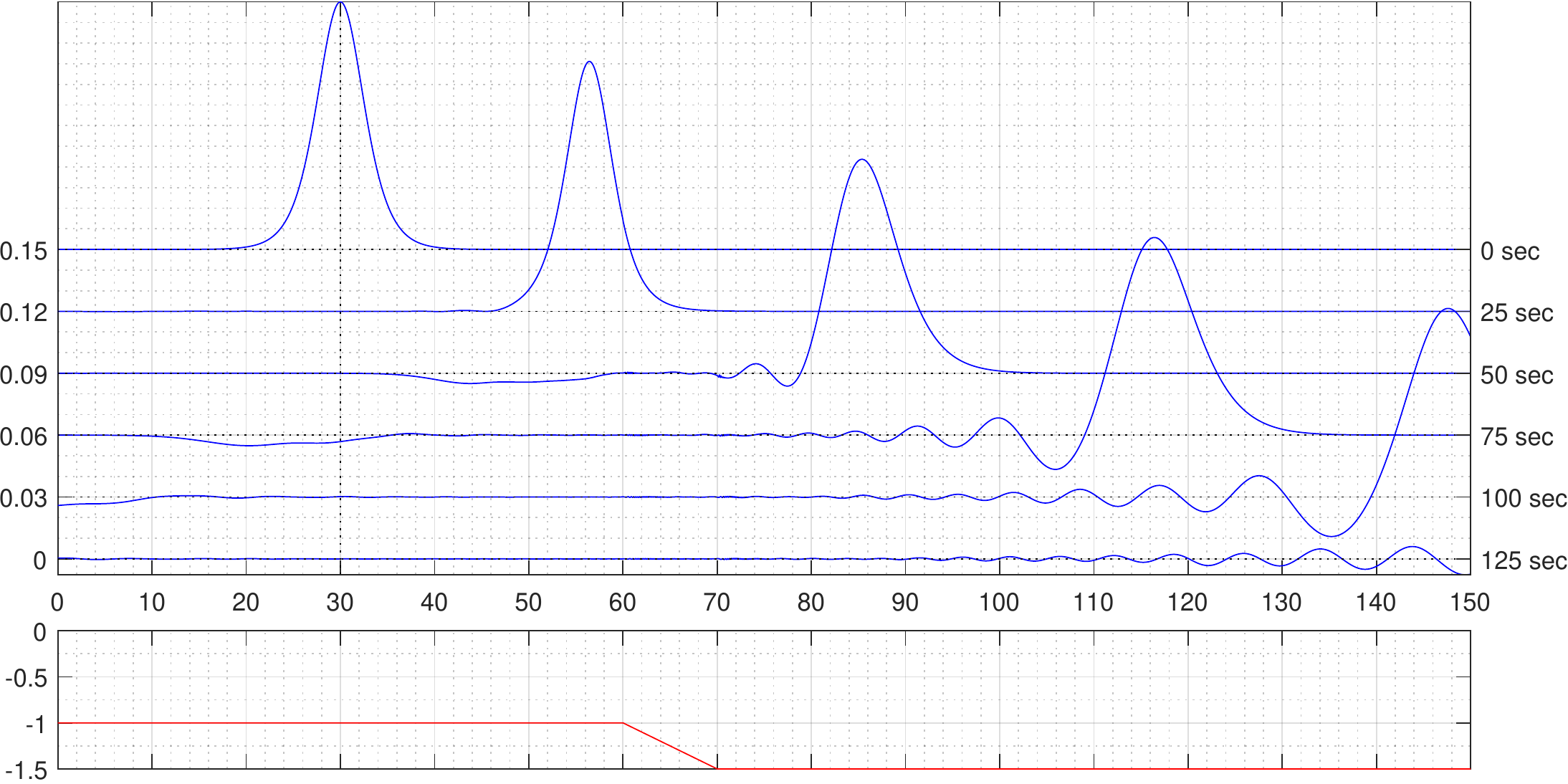}
\vspace{-4ex}
\subcaption{\label{fig:10a}}
\end{subfigure} 
\begin{subfigure}{.9\textwidth} 
\centering 
\includegraphics[width=\textwidth]{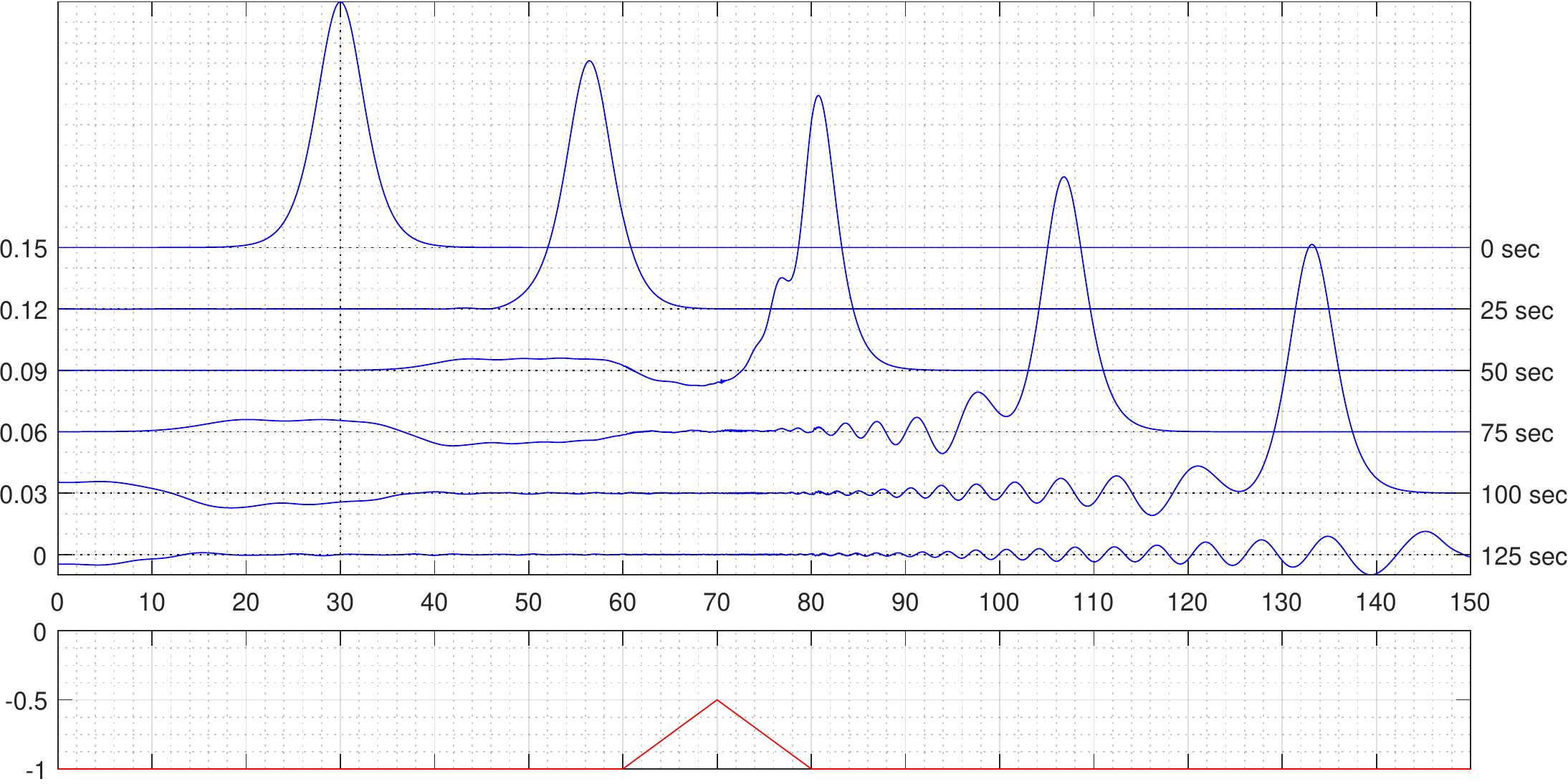}
\vspace{-4ex}
\subcaption{\label{fig:10b}}
\end{subfigure} 
\begin{subfigure}{.9\textwidth} 
\centering 
\includegraphics[width=\textwidth]{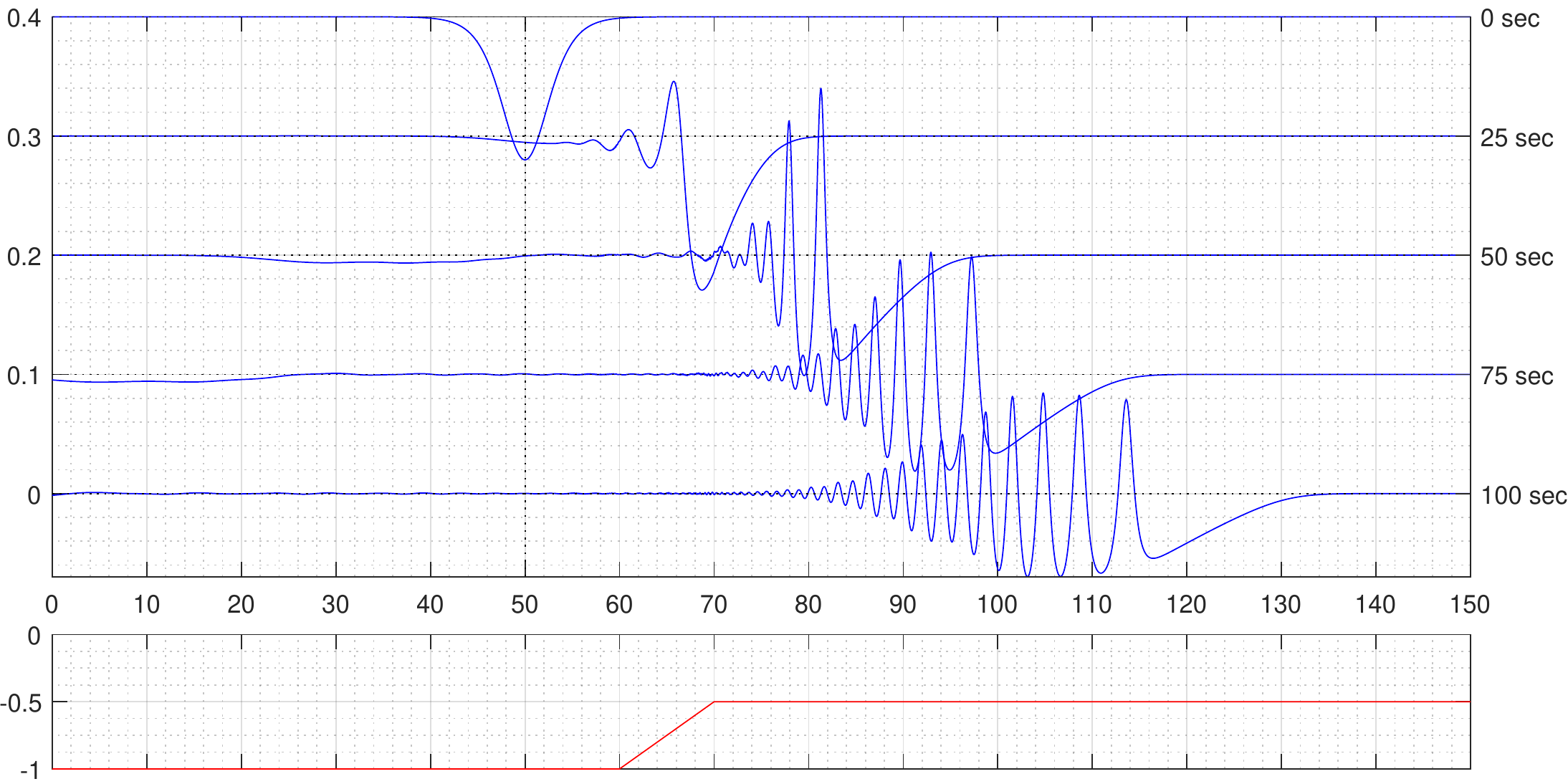}
\vspace{-4ex}
\subcaption{\label{fig:10c}}
\end{subfigure} 
\caption{Reflection due to sloping bottom, various topographies. $\zeta(x,t)$ as 
a function of $x$ at various $t$. \subref{fig:10a}:\ solitary wave travelling 
into deeper water, \subref{fig:10b}:\ solitary wave passing over a hump, 
\subref{fig:10c}:\ wave of depression travelling into shallower water. 
\label{fig:10}}
\end{figure} 
experiments, the results of some of which appear in Figure~\ref{fig:6}. In each 
case an initial wave, originally on a horizontal bottom, is let to evolve under 
\eqref{eq:CBs} and travel over uneven bottoms of various simple topographies 
shown in the lower graphs in Figure~\ref{fig:10}. Fig.~\ref{fig:10a} shows a CB 
solitary wave of amplitude $a_0=0.12$ passing into shallower water. The 
resulting reflected wave is a wave of depression; this solitary wave seems to be 
dispersing as a result of its interaction with the bottom. In the case of a hump 
(Fig.~\ref{fig:10b}) the same initial wave gives rise first to a reflected wave 
of elevation followed by a reflected wave of depression as one would expect. 
This particular perturbation due to this bottom topography seems to lead to a 
solitary wave very close to the initial one plus a trailing dispersive tail. 
Finally, an initial wave of depression climbing upslope gives rise to a 
reflected wave of depression and large-amplitude dispersive oscillations as it 
travels on the shelf.

\subsubsection{Comparison of (CBs) and (CBw) as the variation of the bottom 
increases} 
\label{subsubsec:3p3p4} 

As was mentioned in the Introduction \eqref{eq:CBs} is valid as a model for 
bottoms where topography, described by $\eta_b(x) = 1-\beta b(x)$, may vary 
arbitrarily (so that $\eta_b>0$ of course), i.e.\ where the parameter $\beta$ 
can be taken as an $\mathcal O(1)$ quantity, while \eqref{eq:CBw} was derived 
under the assumption that $\beta=\mathcal O(\varepsilon)$. In this subsection we 
suppose that the systems are written in scaled, dimensionless variables with 
$\mu=\varepsilon$ and we compare computationally the behavior of an initial CB 
solitary wave as it evolves according to each of the two systems travelling over 
a bottom of smooth topography with a fixed shelf-like function $b(x)$ and a 
parameter $\beta$ that varies from $\mathcal O(\varepsilon)$ to $\mathcal O(1)$, 
so that the bottom becomes steeper. 

For this purpose we solve both systems with our fully discrete scheme using 
cubic splines with uniform mesh, $N=2000$ and the RK4 with $M=2N$ on a spatial 
interval of $[0,140]$ with a CB solitary wave of amplitude $0.5$ as initial 
condition. (We experimented with several values of $\varepsilon=\mu$ but the 
results were qualitatively similar, so we show in Figure 
\ref{fig:11} below only the case $\varepsilon=\mu=0.05$.) 

As $b(x)$ we took a fixed profile given by 
\begin{equation}\label{eq:3p10} 
b(x) = \begin{cases} 
0, & x\in \left[0,L-\tfrac 3 2\right], \\ 
\tfrac 1 2 \left(1+\sin\left(\tfrac \pi 3 (x-L)\right)\right),& 
x\in\left[L-\tfrac 3 2, L+\tfrac 3 2\right], \\ 
1,& x\in\left[L+\tfrac 3 2, 140\right], 
\end{cases} 
\end{equation} 
with $L=70$. Thus $b$ is a $C^1$ nonnegative function that bridges $0$ and $1$ 
over an interval of length $3$. As a result, the undisturbed water depth 
$\eta_b$ will vary from $1$ to a shelf of depth $1-\beta$ smoothly over this 
interval. We consider three cases: $\beta=\varepsilon=0.05$, $\beta=0.4$, 
$\beta=0.6$, and present the results of the evolution for $0\leq t\leq 89$ in 
Figure~\ref{fig:11}. In Fig.~\ref{fig:11a}, where $\beta=\varepsilon=0.05$, 
there is, as expected, practically no difference between the two solitary waves 
that suffer only a very small perturbation due to the bottom. But for 
$\beta=\mathcal O(1)$, i.e.\ when the bottom is steeper, we observe in 
Figure~\ref{fig:11b} ($\beta=0.4$) and \ref{fig:11c} ($\beta=0.6$) large 
differences in the solutions of the two systems. As it travels on the shelf the 
solitary wave evolves under \eqref{eq:CBs} into a sequence of solitary waves as 
expected, whilst no such resolution is discernible in the case of the evolution 
under \eqref{eq:CBw} at least for the time frame of this experiment. Both 
systems produce he same small-amplitude reflection waves. Our conclusion is that 
for $\beta=\mathcal O(1)$ \eqref{eq:CBw} does not seem to give the correct 
longer-time behaviour of solutions in the case of strongly varying bottoms. 
\begin{figure}[hptb] 
\centering 
\begin{subfigure}{.9\textwidth} 
\centering 
\includegraphics[width=\textwidth]{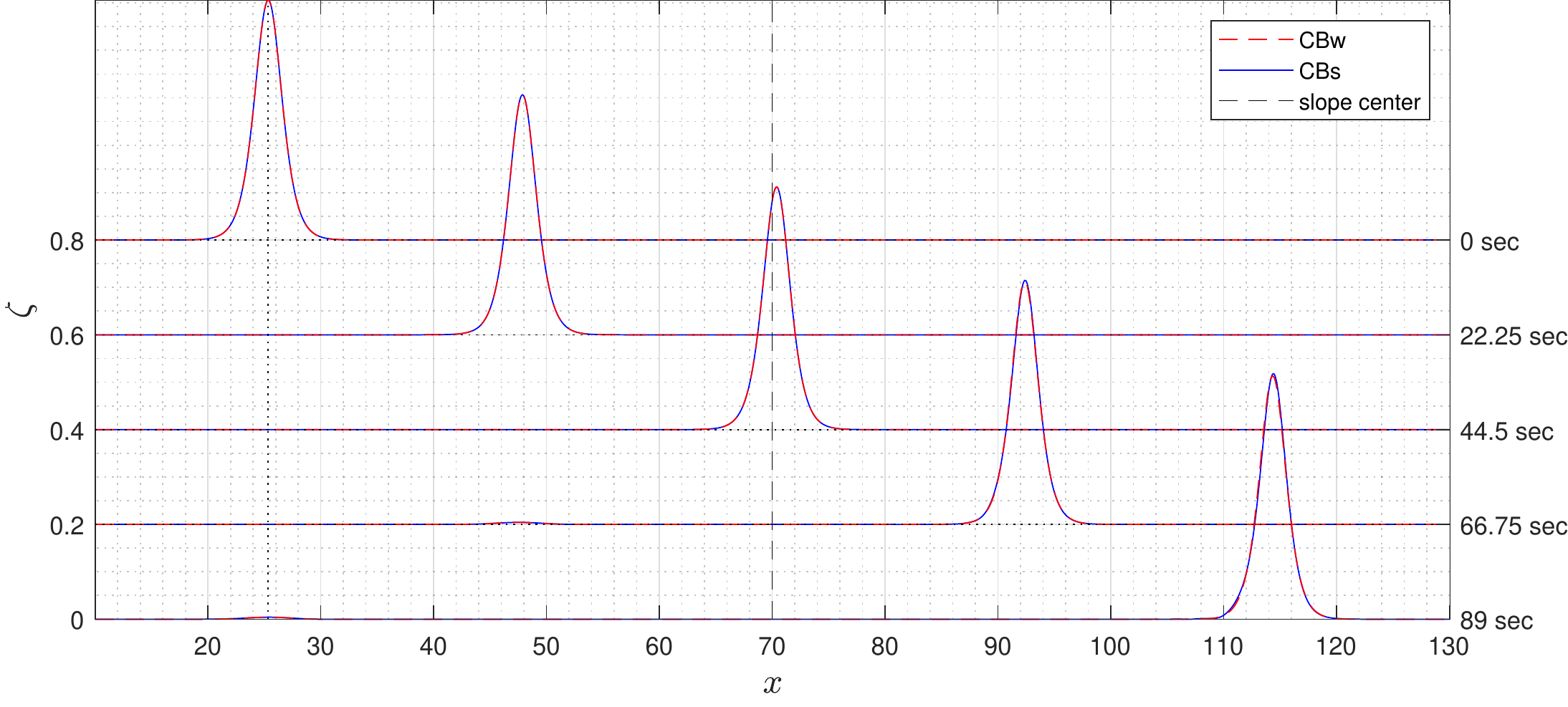}
\vspace{-4ex}
\subcaption{$\beta=\varepsilon=0.05$ \label{fig:11a}}
\end{subfigure}\par
\medskip 
\begin{subfigure}{.9\textwidth} 
\centering 
\includegraphics[width=\textwidth]{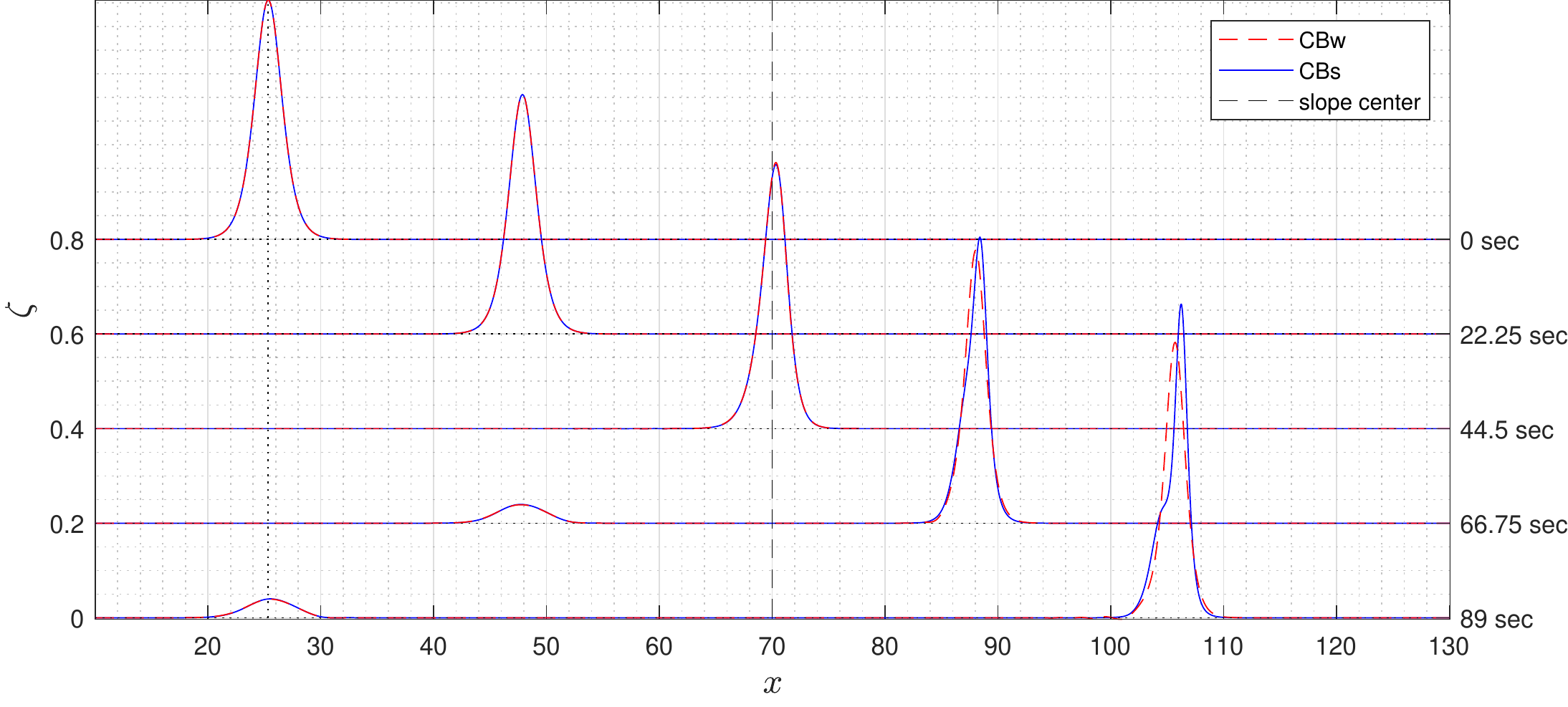}
\vspace{-4ex}
\subcaption{$\beta=0.4$ \label{fig:11b}}
\end{subfigure}\par 
\medskip 
\begin{subfigure}{.9\textwidth} 
\centering 
\includegraphics[width=\textwidth]{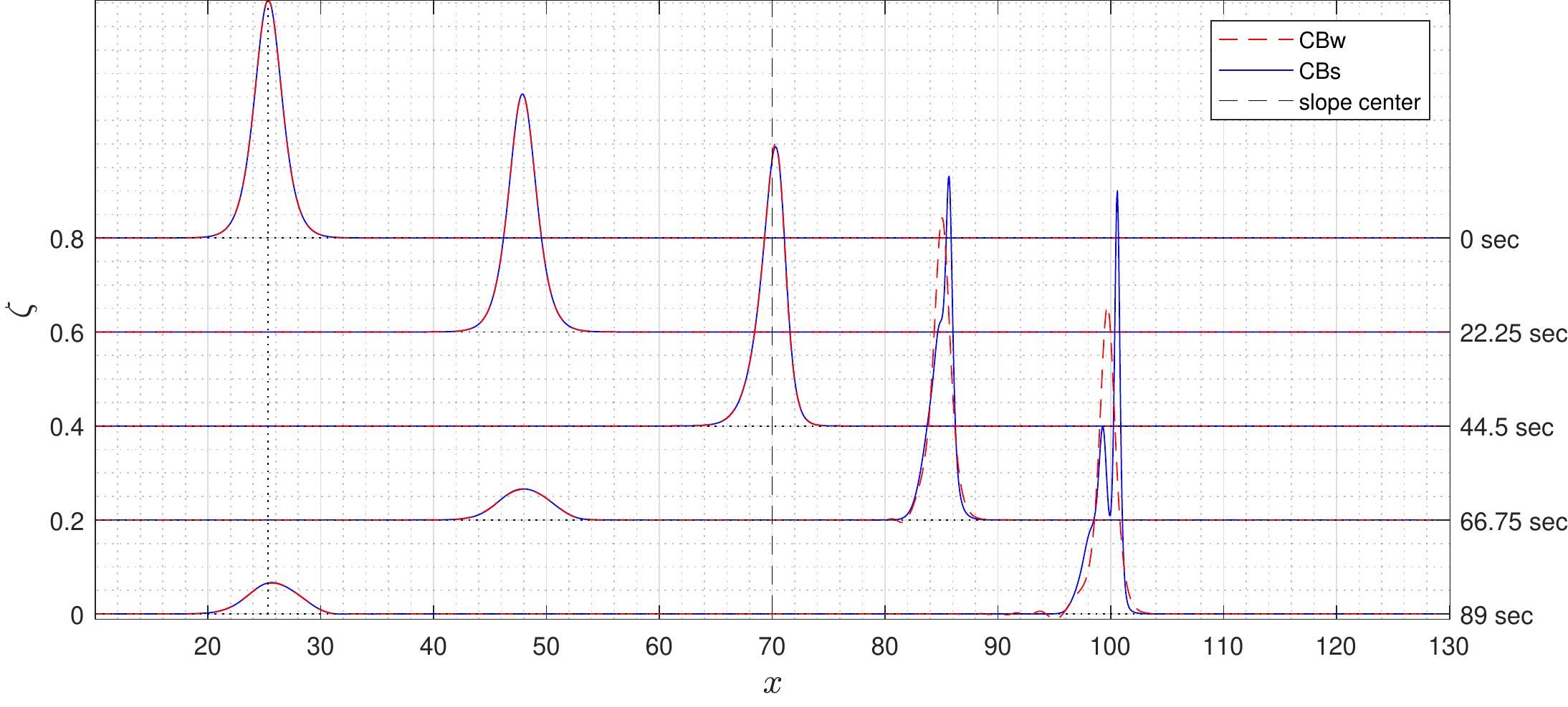}
\vspace{-4ex}
\subcaption{$\beta=0.6$ \label{fig:11c}}
\end{subfigure} 
\caption{Comparison of evolution of a solitary wave under \eqref{eq:CBw} and 
\eqref{eq:CBs} over a bottom of varying steepness: $\eta_b=1-\beta b(x)$, $b(x)$ 
given by \eqref{eq:3p10}, $\varepsilon=\mu=0.05$. \label{fig:11}}
\end{figure}

\subsubsection{Comparison of (CBs) with the Serre-Green-Naghdi system} 
\label{subsubsec:3p3p5} 

Finally, we compare by means of numerical experiment, the evolution of an 
initial solitary wave as it climbs a sloping bed, and as it is reflected by a 
vertical wall at the end of a slope. Recall from the Introduction that the 
system of Serre-Green-Naghdi \eqref{eq:SGN} equations models two-way propagation 
of long dispersive waves (i.e.\ for which $\mu\ll 1$) without the assumption of 
small amplitude, i.e.\ with no restriction of $\varepsilon$, and that 
\eqref{eq:CBs} is obtained from the \eqref{eq:SGN} system with variable bottom 
under the Boussinesq scaling $\varepsilon=\mathcal O(\mu)$, \cite{LB}. The SGN 
system has been used in many computations, cf.\ e.g.\ \cite{CBM2}, \cite{BCLMT}, 
\cite{MSMc}, and their references, that agree quite well with experimental 
results of long-wave propagation over variable bottoms. In \cite{ADM17}, two of 
the authors of the paper at hand, together with D.~Mitsotakis, have analyzed 
Galerkin-finite element methods for \eqref{eq:SGN} on a horizontal bottom (i.e.\ 
for the Serre equations) and shown optimal-order, $L^2$-error estimates in the 
case of periodic splines ($r\geq 3$) on uniform meshes. 

Our aim in this subsection is to compare the results of numerical simulations of 
two test problems with \eqref{eq:CBs}, computed with our code, with numerical 
results for \eqref{eq:SGN} obtained by Mitsotakis {\em et al}.\ in \cite{MSMc}. 
The spatial semidiscretization used in \cite{MSMc} is based on a modified 
Galerkin finite element scheme that uses a projection of a term containing a 
second-order derivative in SGN so that the scheme is also well defined for 
piecewise linear continuous elements (i.e.\ for $r=2$) as well. In what follows 
we will solve numerically \eqref{eq:CBs} using cubic splines on a uniform mesh 
with $N=2000$ and RK4 time stepping with $M=2N$. All variables for this 
experiment are nondimensional and unscaled with $\varepsilon=\mu=1$. 

In the first experiment (shoaling of a solitary wave) we consider the 
variable-bottom example in \S 4.1 of \cite{MSMc}. The geometry, in our notation, 
consists of a channel in the interval $[0,84]$. The bottom is horizontal at a 
depth equal to $-1$ for $0\leq x\leq x_\text{B}=50$, and upsloping with slope 
$\alpha=1/35$ up to $x=84$ where the water depth is equal to $1/35$. The initial 
condition is a solitary wave of the form \eqref{eq:3p6}, \eqref{eq:3p9} of 
amplitude $a_0=0.2$ with crest at $x_0=29.8829$. The evolution of the numerical 
solution is monitored at ten gauges numbered 0, 1, \dots, 9, and located, 
respectively, at $x=45,\ 70.96,\ 72.55,\ 73.68,\ 74.68$, and $76.91$. In this 
experiment the variables are dimensionless and unscaled with 
$\varepsilon=\mu=1$. In the experimental data and the \eqref{eq:SGN} 
computations $g$ was equal to $1$. In Figure~\ref{fig:12} we show the elevation 
of the wave at gauge 0 (at $x=x_\text{B}-5=45$, i.e.\ on the left of the toe of 
the slope), as a function of $t$. The three graphs shown correspond to the 
numerical solutions of \eqref{eq:CBs} and \eqref{eq:SGN}, and to experimental 
data for this problem due to Grilli et al.\ \cite{Gr}, and are all in 
satisfactory agreement. Figure~\ref{fig:13} shows the corresponding graphs of 
the elevation of the wave as a function of time recorded at gauges 1, 3, 5, 7, 
and 9 on the sloping bed. The numerical solution of \eqref{eq:SGN} is in good 
agreement with the experimental data of \cite{Gr}. As the wave climbs up the 
slope the \eqref{eq:CBs} solution grows to a higher amplitude, whose ratio to 
the amplitude of the \eqref{eq:SGN} wave increases monotonically from $1.14$ for 
gauge 1 to $1.49$ for gauge 9. 

\begin{figure}[htbp] 
\centering 
\includegraphics[width=.9\textwidth]{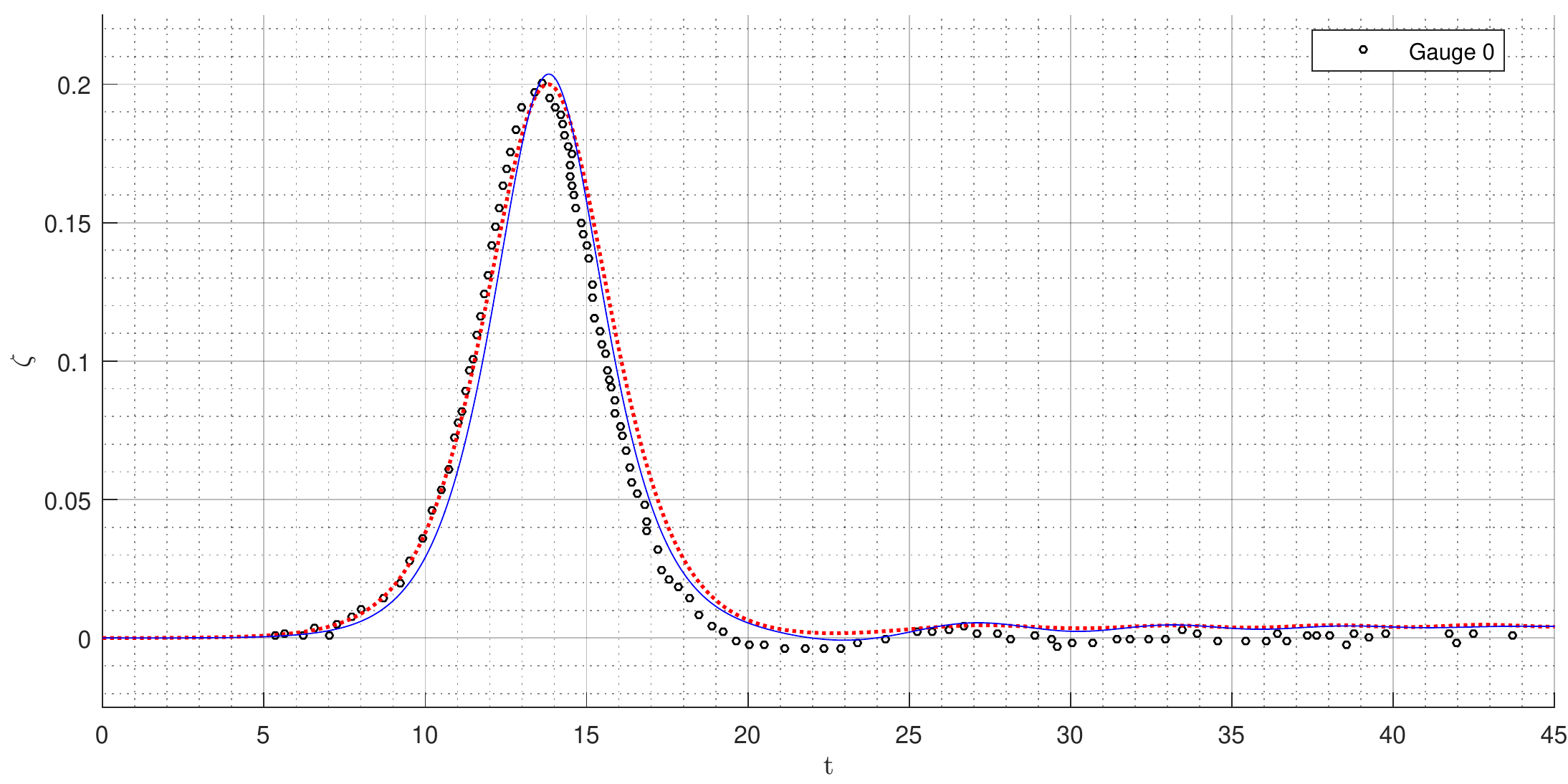}
\caption{Elevation of wave at $x=x_\text{B}-5=45$ as a function of time. Markers 
show the experimental data, \cite{Gr}, dotted lines the numerical solution of 
\eqref{eq:SGN}, \cite{MSMc}, and solid lines the numerical solution of 
\eqref{eq:CBs} system. \label{fig:12}} 
\end{figure} 

\begin{figure}[htbp] 
\centering 
\includegraphics[width=.9\textwidth]{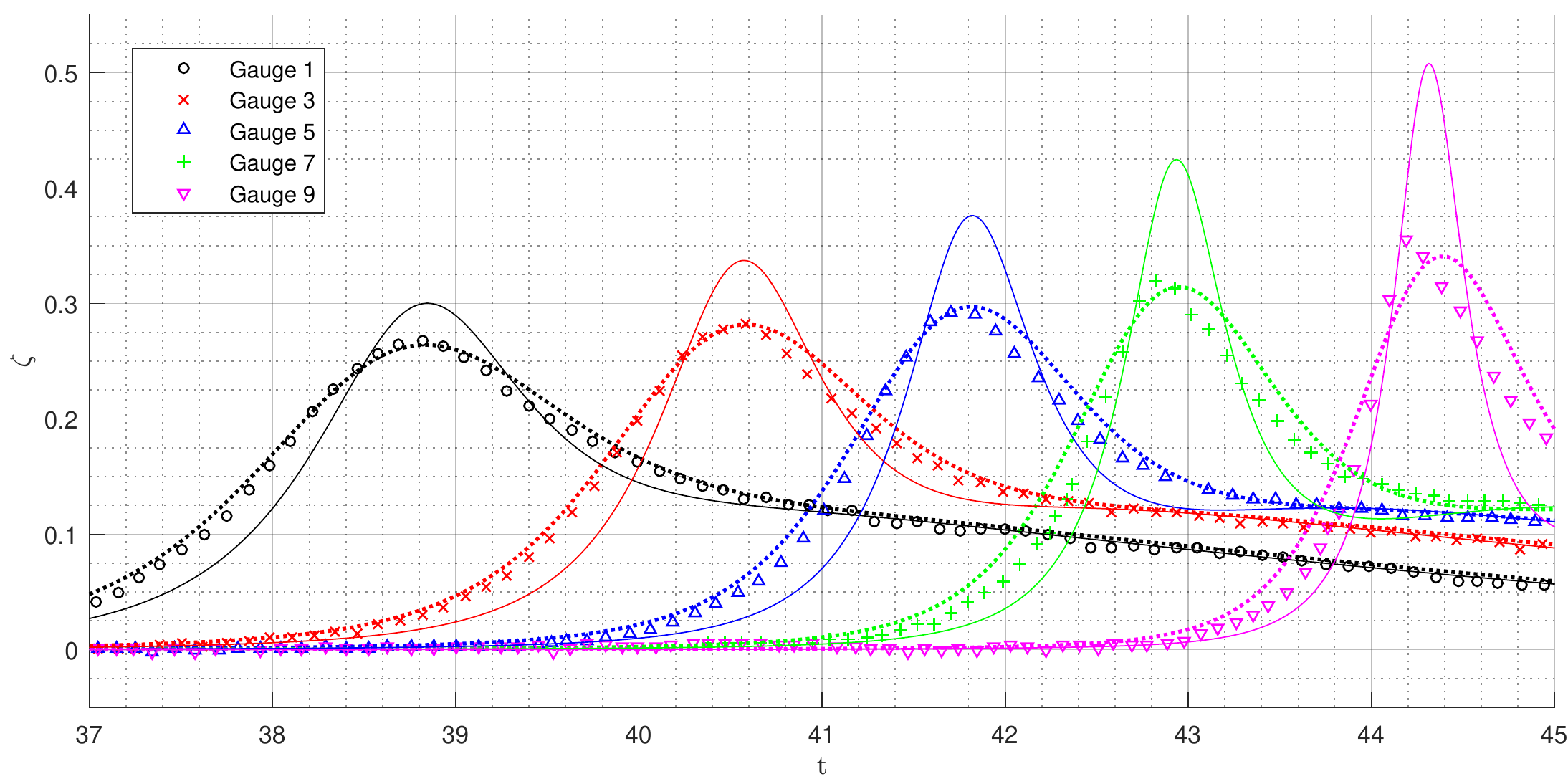}
\caption{Elevation of wave at various gauges as a function of time for the 
shoaling on a beach of slope $1:35$ of a solitary wave with $a_0=0.12.$ Markers 
show experimental daa, \cite{Gr}, dotted lines the numerical solution of 
\eqref{eq:SGN}, \cite{MSMc}, and solid lines the numerical solution of 
\eqref{eq:CBs} system. \label{fig:13}} 
\end{figure} 

For the second numerical experiment (shoaling and reflection of a solitary wave 
from a vertical wall at the end of the sloping beach), we consider a benchmark 
problem, cf.\ e.g.\ \cite{MSMc}, \cite{WB}, \cite{D}, \cite{CBM2}, \cite{BCLMT}, 
among other, that we solve numerically with our code of \eqref{eq:CBs} and 
compare the results with those found by the numerical integration of 
\eqref{eq:SGN} in Section~4.3 of \cite{MSMc}, and with experimental data due to 
Dodd, \cite{D}. The setup consists of a channel of length $[0,70]$, initially 
horizontal at a depth of $h_0=0.7$, a sloping bed of uniform slope $1:50$ that 
starts rising at $x_\text{B}=50$ and ends at $x=70$, where a vertical wall is 
placed. (This is shown in the lower graph of Figure~\ref{fig:14}.) We consider 
two solitary waves of the form \eqref{eq:3p6}, \eqref{eq:3p9} (suitably modified 
so that the horizontal part of the waveguide has now a depth of $h_0=0.7$) with 
amplitudes $0.07$ and $0.12$ and crest initially located at $x=20$. We solve the 
problem numerically with our code for \eqref{eq:CBs} with a boundary condition 
$u=0$ using cubic splines, $N=2000$, $M=2N$. All variables for this experiment 
are dimensional, $x$ and $\eta$ are measured in meters and $t$ in seconds. The 
parameters $\varepsilon$ and $\mu$ are equal to $1$. The value of the 
gravitational acceleration constant is $g=9.80665\,\mathrm m/\mathrm s^2$ 
(standard gravity). 

In Figure~\ref{fig:14} we how snapshots every 3 secs of the \eqref{eq:CBs}-free 
surface elevation as a function of $x$ as the wave (of initial amplitude 
$a_0=0.07$) climbs up the slope and is reflected by the wall at $x=70$ between 
$t=15$ and $t=18$. The reflected pulse apparently consists 
\begin{figure}[htbp] 
\centering 
\includegraphics[width=\textwidth]{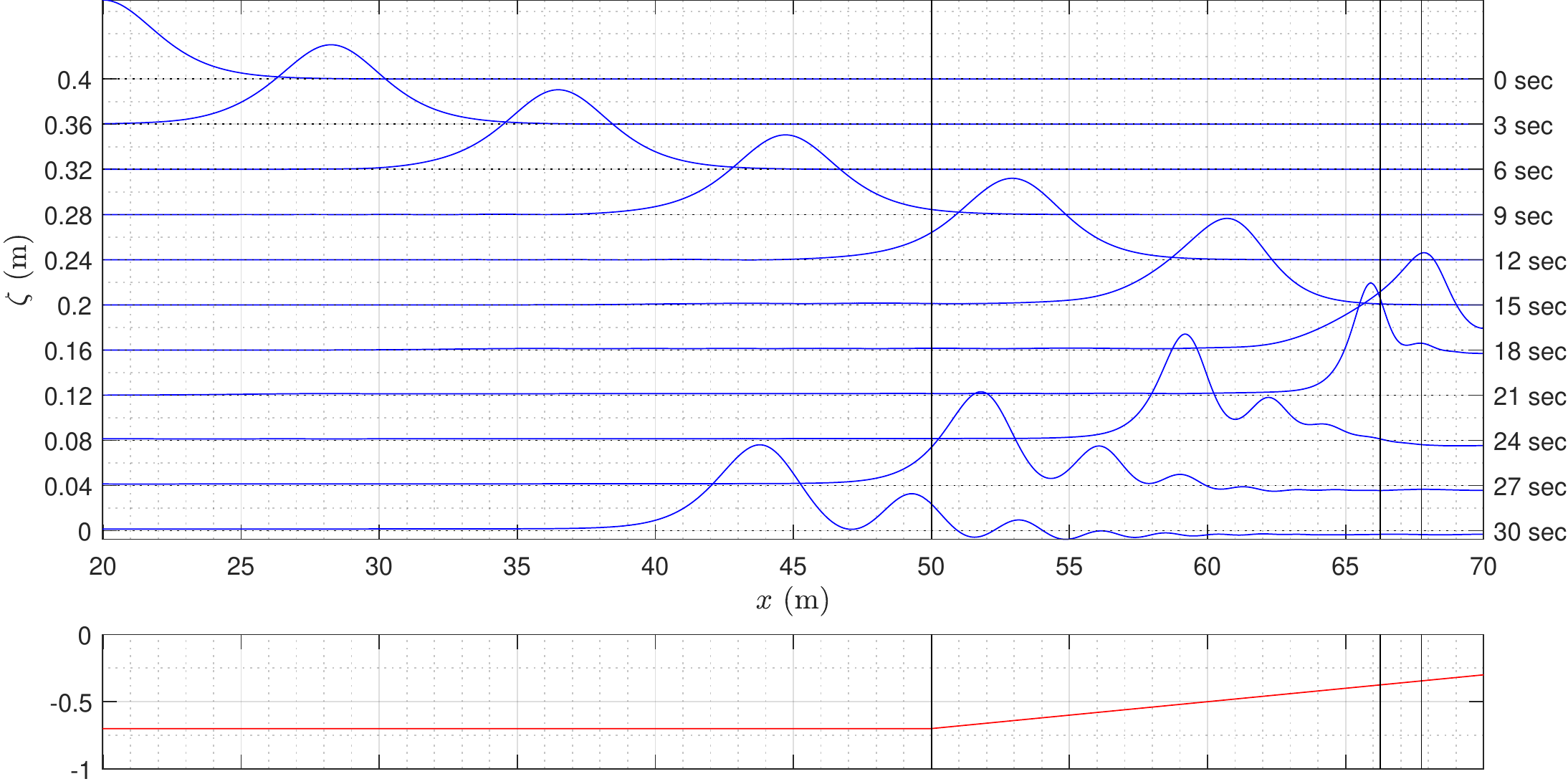} 
\caption{Evolution of the solitary wave of amplitude $a_0=0.07$ according to 
\eqref{eq:CBs} on a beach of slope $1:50$, reflected on a vertical wall at 
$x=70$. Vertical lines depict the location of gauges 1, 2 and 3. \label{fig:14}} 
\end{figure} 
of a leading pulse followed by a dispersive tail. This wave travels downslope, 
and by $t=30$ the leading pulse is located well within the horizontal-bottom 
region. The maximum runup at the wall was recorded to be equal to $.1899$. 

In the (related) Figure~\ref{fig:15} we show the temporal histories of the wave 
elevation $\zeta(x,t)$, generated by the solitary wave of amplitude $a_0=0.07$, 
at three gauges $g_1$, $g_2$, $g_3$, located at $x=50$, $x=66.25$, and $x=67.75$ 
(very close to the wall), respectively, computed by \eqref{eq:CBs} and 
\eqref{eq:SGN} (code of \cite{MSMc}), in comparison with the experimental data 
of \cite{D} for this problem. 

\begin{figure}[hptb] 
\centering 
\includegraphics[width=\textwidth]{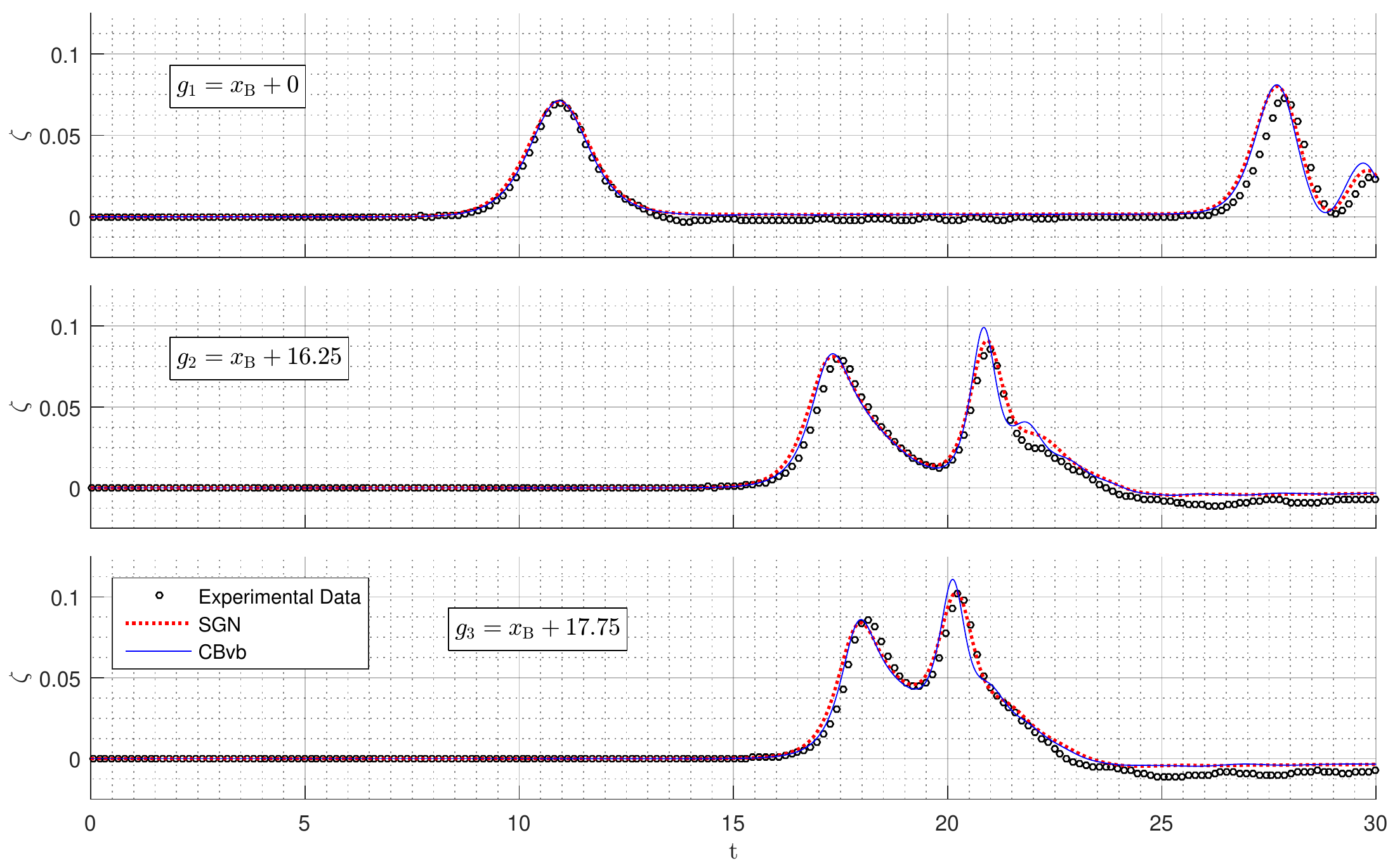} 
\caption{Reflection at a vertical wall located at $x=70$ of a shoaling wave over 
a beach of slope $1:50$, with toe at $x_\text{B}=50$. Initial solitary wave 
amplitude $a_0=0.07$. \label{fig:15}} 
\end{figure} 

\begin{figure}[hptb] 
\centering 
\includegraphics[width=\textwidth]{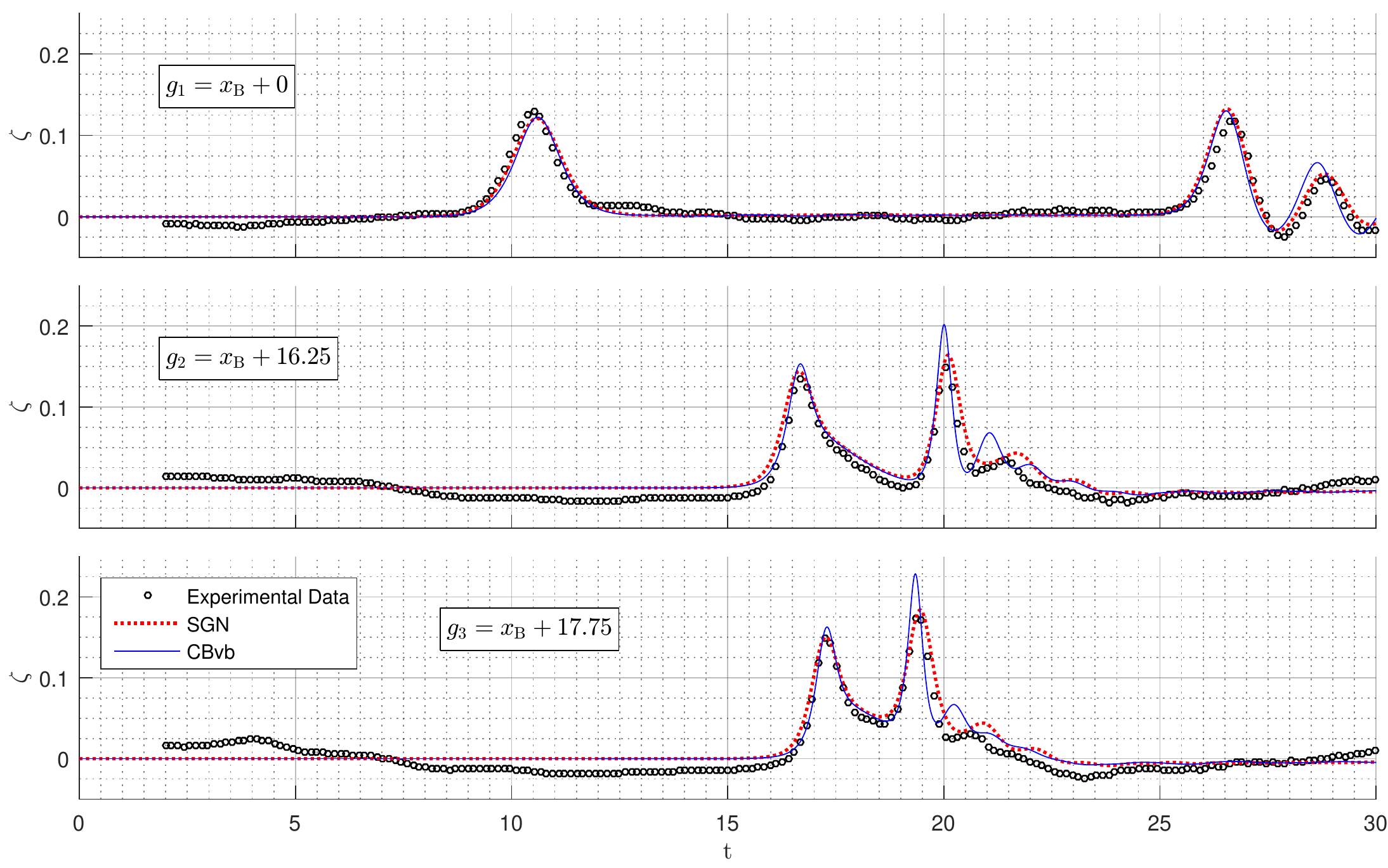} 
\caption{Reflection at a vertical wall located at $x=70$ of a shoaling wave over 
a beach of slope $1:50$, with toe at $x_\text{B}=50$. Initial solitary wave 
amplitude $a_0=0.12$. \label{fig:16}} 
\end{figure} 

We observe that there is quite a good agreement between the three curves. The 
maximum amplitude of the reflected wave at gauge $g_3$ is found to be equal to 
$.11080$ for \eqref{eq:CBs} and to $.10280$ for \eqref{eq:SGN}, giving a ratio 
of about $1.08$. 

Figure~\ref{fig:16} depicts the analogous graphs in the case of the initial 
solitary wave of amplitude $a_0=0.12$. (Note the different scale of the 
$\zeta$-axis.) This wave becomes steeper as it climbs up the slope; the 
reflected wave is of higher amplitude as well. The incident waves computed by 
the two models are quite close to each other and to the experimental data but 
the short-time behavior of the reflected pulse is somewhat different. For 
example, at $g_3$ the amplitude of the reflected \eqref{eq:CBs} pulse is now 
equal to $.2285$ while the amplitude of the \eqref{eq:SGN} reflected pulse is 
$.1838$ (giving a ratio of about $1.24$), and there are phase and amplitude 
differences in the leading trailing oscillations. When the reflected wave has 
returned to the horizontal part of the channel (i.e.\ at $g_1$ in 
Figure~\ref{fig:16} for $t\geq 25$) the agreement is much better and the ratio 
is now $0.98$. The leading reflected pulse of the \eqref{eq:SGN} solution is in 
satisfactory agreement with the data at all three gauges. The maximum runup at 
the wall of \eqref{eq:CBs} for this amplitude was equal to $.4012$. 

Our conclusion from the two numerical experiments in this subsection is that 
when the elevation wave steepens either while climbing up a sloping beach or 
after reflection from a vertical wall and close to the wall, the \eqref{eq:CBs} 
solution overestimates that of the \eqref{eq:SGN}; the latter stays quite close 
to the available experimental data in the cases that we tried.

\section*{Acknowledgements} 
This research was partially supported by IACM-FORTH by the grant ``Innovative 
Actions in Environmental Research and Development (PErAn)'' (MIS 5002358), 
implemented under the ``Action for the strategic development of the Research and 
Technological sector'' funded by the Operational Program ``Competitiveness, and 
Innovation'' (NSRF 2014-2020) and cofinanced by Greece and the EU (European 
Regional Development Fund). G.~Kounadis also acknowledges scholarship support in 
the initial stages of the project from the Stavros Niarchos Foundation `Archers' 
grant to FORTH. The authors also express their thanks to Dr.\ D.~E.~Mitsotakis 
for making available to them the numerical data for \eqref{eq:SGN} of 
\cite{MSMc} quoted in the last two experiments.

\bibliographystyle{elsarticle-num} 
\bibliography{cbvb_cg_arxiv}

\end{document}